 \documentclass[12pt]{amsart} 
 \usepackage{amsmath} \usepackage{amsxtra}
 \usepackage{amscd} \usepackage{amsthm}  
 \usepackage{amsfonts}\usepackage{amssymb} 
  \usepackage{eucal} 
 \usepackage{tikz}
 \usetikzlibrary{cd} 
 \usepackage{enumitem}
 \usepackage{xcolor}
 \usepackage{verbatim}
 \usepackage{array}
 \usepackage{textcomp}
 \usepackage{hyperref}
 \usepackage{colonequals}  
 \usepackage{graphicx}
\usepackage{tabularx}

 \usepackage[francais]{babel}
 \usepackage{lmodern}
\usepackage{dirtytalk}
 \usepackage[utf8]{inputenc}
 \usepackage[T1]{fontenc} 
 \newtheorem{Cor}[subsubsection]{Corollaire}
\newtheorem{Lm}[subsubsection]{Lemme}
\newtheorem{Pp}[subsubsection]{Proposition}
\newtheorem{Conj}[subsubsection]{Conjecture}
\newtheorem{Thm}[subsubsection]{Théorème}
\newtheorem{Def}[subsubsection]{Définition}
\newtheorem{Rq}[subsubsection]{Remarque}

 \newtheorem{Q}[subsubsection]{Question}
  \newtheorem{Hyp}[subsubsection]{Hypothèse}
 \newtheorem*{Conj*}{Conjecture}
 \theoremstyle{remark}
 
\newlist{lenumerate}{enumerate}{1}
\setlist[lenumerate,1]{label=\roman*), ref=\theLm(\roman*), leftmargin=0em, itemindent=1.5\parindent}
 \newcommand{\nc}{\newcommand}
 \nc{\rnc}{\renewcommand}
 \nc{\on}{\operatorname}
 \rnc{\sec}{\section}
 \nc{\ssec}{\subsection}
 \nc{\sssec}{\subsubsection}
 \rnc{\theenumi}{\roman{enumi})}
\rnc{\labelenumi}{\theenumi}

 \nc{\cA}{{\mathcal A}}
 \nc{\cB}{{\mathcal B}}
 \nc{\cC}{{\mathcal C}}
 \nc{\cD}{{\mathcal D}}
 \nc{\cE}{{\mathcal E}}
 \nc{\cF}{{\mathcal F}}
 \nc{\cG}{{\mathcal G}}
 \nc{\cH}{{\mathcal H}}
 \nc{\cI}{{\mathcal I}}
 \nc{\cJ}{{\mathcal J}}
 \nc{\cK}{{\mathcal K}}
 \nc{\cL}{{\mathcal L}}
 \nc{\cM}{{\mathcal M}}
 \nc{\cN}{{\mathcal N}}
 \nc{\cO}{{\mathcal O}}
 \nc{\cP}{{\mathcal P}}
 \nc{\cQ}{{\mathcal Q}}
 \nc{\cR}{{\mathcal R}}
 \nc{\cS}{{\mathcal S}}
 \nc{\cT}{{\mathcal T}}
 \nc{\cU}{{\mathcal U}}
 \nc{\cV}{{\mathcal V}}
 \nc{\cW}{{\mathcal W}}
 \nc{\cX}{{\mathcal X}}
 \nc{\cY}{{\mathcal Y}}
 \nc{\cZ}{{\mathcal Z}}
 \rnc{\AA}{{\mathbb A}}
 \nc{\BB}{{\mathbb B}}
 \nc{\CC}{{\mathbb C}}
 \nc{\DD}{\mathbb{D}}
 \nc{\EE}{{\mathbb E}} 
 \nc{\FF}{{\mathbb F}}
 \nc{\GG}{{\mathbb G}}
 \nc{\HH}{{\mathbb H}}
 \nc{\KK}{{\mathbb K}}
 \nc{\LL}{{\mathbb L}}
 \nc{\NN}{{\mathbb N}}
 \nc{\OO}{{\mathbb O}}
 \nc{\PP}{{\mathbb P}}
 \nc{\QQ}{{\mathbb Q}}
 \nc{\RR}{{\mathbb R}}
 \nc{\SSS}{{\mathbb S}}
 \nc{\TT}{{\mathbb T}}
 \nc{\VV}{{\mathbb V}}
 \nc{\WW}{{\mathbb W}}
 \nc{\YY}{{\mathbb Y}}
 \nc{\ZZ}{{\mathbb Z}}
 \nc{\frg}{\mathfrak{g}} 
 \nc{\frm}{\mathfrak{m}} 
 \nc{\frG}{\mathfrak{G}} 
 \nc{\frF}{\mathfrak{F}} 
 \nc{\frS}{\mathfrak{S}} 
 
 \nc{\rmd}{\mathrm{d}} 
 \nc{\rmK}{\mathrm{K}} 
 \nc{\rmR}{\mathrm{R}} 
 \nc{\rmG}{\mathrm{G}} 
 \nc{\rmF}{\mathrm{F}}

 \rnc{\a}{{\alpha}}
 \rnc{\b}{{\beta}}
 \nc{\g}{{\gamma}}
\nc{\Div}{\on{Div}}
 \nc{\Pic}{\on{Pic}}

\nc{\Conv}{\on{Conv}} 
\nc{\Vol}{\mathrm{Vol}}
\nc{\Gal}{\on{Gal}}
 \nc{\HHH}{\on{H}} 
 \nc{\Heis}{\on{H}} 
  \rnc{\P}{\on{\mathbb{P}}} 
  \nc{\Tr}{\on{Tr}} 
   \nc{\pr}{\on{pr}} 
  \nc{\ocirc}[1]{\overset{\circ}{#1}}
  \nc{\Ouv}[1]{\ocirc{#1}}
  \nc{\dual}[1]{{#1}^{\vee}}
  \rnc{\Im}{\on{Im}} 
 \nc{\Bun}{\on{Bun}} 
 \nc{\Hom}{\on{Hom}} 
 \nc{\Ker}{\on{Ker}}
 \nc{\Aut}{{\on{Aut}}} 
 \nc{\colim}{\on{colim}} 
 
 \nc{\supp}{\on{supp}} 
 
 \nc{\IC}{\on{IC}} 
 \nc{\ST}{{\on{ST}}} 
 \nc{\Sym}{{\on{Sym}}}  
 \nc{\Nm}{\on{Nm}}  
 \nc{\ev}{\on{ev}} 
 \nc{\e}{\mathrm{e}} 
 \nc{\id}{\mathrm{id}} 
  \nc{\ass}[2]{{{#1}_{\cF_{#2}}}} 
  \nc{\ts}[3]{{{#1}\ot_{#2}{#3}}}
 \nc{\T}{{\on{T}}}
 \nc{\car}{{\on{car}}}
 \nc{\Out}{\on{Out}}
 \nc{\Spin}{\on{Spin}}
 \nc{\Gm}{\mathbb{G}_m}
 \nc{\Mor}{\on{Mor}}
 \nc{\Spec}{\on{Spec}}
 \nc{\dec}[1]{[{#1}](\frac{#1}{2})} 
 \nc{\ndec}[1]{[-{#1}](-\frac{#1}{2})} 
 \nc{\conv}{\ast}
 \nc{\Qb}{\bar{\QQ}}
 \nc{\Ql}{\QQ_{\ell}}
 \nc{\Zl}{\ZZ_{\ell}}
 \nc{\Qlb}{\overline{\QQ_{\ell}}}
 \nc{\inv}{^{-1}}
 \nc{\Four}{\on{Four}}
 \nc{\Res}{\on{Res}}
 \nc{\card}[1]{\on{card}(#1)} 
 
  \nc{\Lie}{\on{Lie}}
  \nc{\Ga}{\GG_a} 
  \nc{\Gr}{\on{Gr}} 
   \nc{\SO}{{\on{S\OO}}}
   \nc{\GL}{{\on{GL}}}
   \nc{\PGL}{\on{PGL}}
   \nc{\SL}{\on{SL}}
   \nc{\Sp}{\on{Sp}}
   \rnc{\O}{\on{O}}
   \nc{\Fq}{\FF_q}
\nc{\rg}{\on{rg}}
\nc{\ad}{\mathrm{ad}}
\nc{\Ad}{\mathrm{Ad}}
\nc{\Swan}{\mathrm{Swan}}

\nc{\Frac}{\on{Frac}}
\nc{\Unit}[1]{{#1}^{\times}}	
 \nc{\f}{{\forall}}
\nc{\os}{\overset}
\nc{\mto}{\mapsto}  
 \nc{\ito}{{\xrightarrow{\sim}}}
 \nc{\xto}{\xrightarrow}
\nc{\hto}{{\hookrightarrow}}
\nc{\wt}{\widetilde}
\nc{\ul}{\underline}
\nc{\Om}{{\Omega}}
\nc{\om}{{\omega}}
\nc{\bra}{\langle}
\nc{\ket}{\rangle}  
\rnc{\t}{\times}
\nc{\lt}{\ltimes}
\nc{\ot}{\otimes}
\nc{\op}{\oplus}
\nc{\coupl}[2]{\langle #1, #2\rangle} 
\nc{\Equiv}{\Leftrightarrow}
\nc{\da}{\downarrow}
 \nc{\epc}{{\check{\epsilon}}}
 \rnc{\phi}{\varphi}
 \nc{\eseq}[3]{0\to{#1}\to{#2}\to{#3}\to 0}

\nc{\Pa}{P_{\a}}
\nc{\Pb}{P_{\b}}

\nc{\Za}{Z_{\a}}
\nc{\Zb}{Z_{\b}}
\nc{\Ma}{M_{\a}}
\nc{\Mb}{M_{\b}}
\nc{\Ub}{U_{\b}}
\nc{\UPa}{\on{R}_u(\Pa)}
\nc{\UPb}{\on{R}_u(\Pb)}
\nc{\Ug}{U_{\g}}
\nc{\YPa}{\cY_{\Pa}}
\nc{\Fsr}{\cF_{sr}}
\nc{\Fred}{\textcolor{red}{\Four}}
\nc{\Fblue}{\textcolor{blue}{\Four}}
\nc{\Fgreen}{\textcolor{green}{\Four}}
\nc{\ocX}{\ocirc{\cX}}
\nc{\ocY}{\ocirc{\cY}}
\nc{\ja}{j_{\a}}
\nc{\jb}{j_{\b}}
\nc{\pia}{\pi_{\a}}
\nc{\pib}{\pi_{\b}}
\nc{\Franel}{\rm{Franel}}
\nc{\Hk}{\rm{Hk}}
\nc{\Sat}{\rm{Sat}}
\nc{\Rep}{\rm{Rep}}
\nc{\EF}{E_{\cF}}
\nc{\VGm}{V|_{\Gm}}
\nc{\Fmin}{\cF_{\text{min}}}
\nc{\Ftriv}{\cF_{triv}}
\nc{\Fst}{\cF_{\text{st}}}
\nc{\Fsgn}{\cF_{sgn}}
\nc{\Frho}{\cF_{\rho}}
\nc{\rhotriv}{\rho_{\text{triv}}}
\nc{\rhost}{\rho_{\text{st}}}
\nc{\rhosgn}{\rho_{\text{sgn}}}
\nc{\Eis}{\rm{Eis}}
\nc{\bBun}{\overline{\Bun}}

\nc{\bU}{{\bar{U}}}
\nc{\bL}{{\bar L}}
\nc{\uL}{{\ul{L}}}
\nc{\tu}{{\tilde{u}}}
\nc{\tR}{\tilde{R}}
\nc{\tM}{\tilde{M}}
\nc{\ttM}{\tilde{\tilde{M}}}
\nc{\ttR}{\tilde{\tilde{R}}}
\nc{\UtR}{U_{\tilde{R}}}
\nc{\bu}{{\bar u}}
\nc{\asp}[3]{{{#2}{\overset{{#1}}{\times}}{#3}}}
\nc{\asss}[3]{{{({#1}_{\cF_{#3}})}_{\cF_{({#2}_{\cF_{#3}})}}}} 
\nc{\assT}[2][T]{{\cF_{(\ass{#2}{#1})}}}
\nc{\cZo}{\ocirc{\cZ}}
\nc{\cYo}{\ocY}
\nc{\WF}{{\cW_{\cF_{P_\alpha}}}}
\nc{\VF}{{\cV_{\cF_{P_\alpha}}}}
\nc{\FPa}{{\cF_{P_\alpha}}}
\nc{\piRR}{{\pi_{\cR'}}}
\nc{\piR}{{\pi_\cR}}
\nc{\ppiRR}{{\pi'_{\cR'}}}
\nc{\ppiR}{{\pi'_\cR}}
\nc{\vastF}{{v^*_{\FPa}}}
\nc{\evR}{{\ev_\cR}}
\nc{\cZP}{{\cZ_P}}
\nc{\dV}{{\det\cV}}
\nc{\RPa}{{\cR_{P_\alpha}}}
\nc{\Piss}{{\Pi^{ss}}}
\nc{\Pis}{{\Pi^s}}
\nc{\phia}{{\phi_\alpha}}
\nc{\phiu}{{\phi_u}}
\nc{\Ubch}{{U_{\check\beta}}}
\nc{\ach}{{\check{\alpha}}}
\nc{\bch}{{\check{\beta}}}
\nc{\gch}{{\check{\gamma}}}
\nc{\FT}{{\cF_T}} 
\nc{\Ua}{{U_\alpha}}
\nc{\UUU}[2][]{{U_{#1\ach+#2\bch}}}
\nc{\Ue}[3][+]{{U_{\epc_{#2}{#1}\epc_{#3}}}}
\nc{\UeT}[2]{{\ass{\Ue{#1}{#2}}{T}}}
\nc{\U}{{\Ubch}}
\nc{\Uu}{{\UUU{2}}} 
\nc{\Ud}{\UUU{}} \nc{\UdY}{{U^Y_{\ach+\bch}}} 
\nc{\Ut}{{U_\ach}}
\nc{\Uq}{{\UUU{3}}} 
\nc{\Uc}{{\UUU[2]{3}}}
\nc{\UF}{{{\U}_{\FT}}}
\nc{\UFu}{{{\Uu}_{\FT}}}
\nc{\UFd}{{{\Ud}_{\FT}}}  \nc{\UFdY}{{{U^Y_{\ach+\bch}}_\FT}}
\nc{\UFt}{{{\Ut}_{\FT}}}
\nc{\UFq}{{{\Uq}_{\FT}}}
\nc{\UFc}{{{\Uc_{\FT}}}}
\nc{\FU}{{\cF_{(\UF)}}}
\nc{\FUu}{{\cF_{(\UFu)}}}
\nc{\FUd}{{\cF_{(\UFd)}}} \nc{\FUdY}{{\cF_{(\UFdY)}}}
\nc{\FUt}{{\cF_{(\UFt)}}}
\nc{\FUq}{{\cF_{(\UFq)}}}
\nc{\ab}[2][]{{{#1}\ach+{#2}\bch}}
\rnc{\ggg}[2][]{{\frg_{#1\ach+#2\bch}}}
\nc{\gb}{{\frg_\bch}} \nc{\gF}{{\gb_\FT}} 
\nc{\gu}{{\ggg{2}}}   \nc{\guF}{{\gu_\FT}}
\nc{\gd}{{\ggg{}}}     \nc{\gdF}{{\gd_\FT}}
\nc{\gt}{{\frg_\ach}}   \nc{\gtF}{{\gt_\FT}}
\nc{\gq}{{\ggg{3}}}     \nc{\gqF}{{\gq_\FT}}
\nc{\gc}{{\ggg[2]{3}}}  \nc{\gcF}{{\gc_\FT}}

\nc{\sog}[3][+]{{\frg_{\epc_{#2}{#1}\epc_{#3}}}} 
\nc{\sogT}[3][+]{{{\sog[#1]{#2}{#3}}_\FT}}
\nc{\tU}{{\wt{U}}} 
\nc{\wedgech}{\check{\wedge}}
\nc{\Phich}{\check{\Phi}}
\nc{\RG}{R_G}

\nc{\tensR}[1]{#1_{\RR}}
\nc{\Hyper}{\rm{Hyp}^0_{\psi}}
\nc{\Sh}[1]{\rm{Sh}(#1)}
\nc{\LR}{\tensR{\Lambda}}
\nc{\cone}[1]{\cC_{#1}}
\nc{\aG}{a_{\Gamma}}
\nc{\aF}{a_{\Phi}}
\nc{\Xa}{\cX_{\alpha}}
\nc{\HX}[1]{\HHH^0(X,#1)} 
\nc{\diff}[1]{(\rm{d} #1)_{\ast}} 
\nc{\fr}[1]{\mathfrak{#1}}
\nc{\frI}{\fr{I}}
\nc{\fri}{\fr{i}}
\nc{\cXi}{\cX_{\pi,\nu,\mu}}
\nc{\phii}{\phi_{\pi,\nu,\mu}}

\nc{\RGc}{\on{R\Gamma_c}}
\nc{\PhiD}[2]{\Phi_{\Delta}(#1; #2)}
\nc{\PD}{P_{\Delta}}
\nc{\Lpsi}{\cL_{\psi}}
\nc{\Dinf}{\Delta_{\infty}}
\nc{\LtR}{\Lambda\otimes\RR}
\nc{\dLtR}{\dual{\Lambda}\otimes\RR}
\nc{\Autg}{\Aut_{\mathrm{grp}}}
\nc{\Autv}{\Aut_{\mathrm{var}}}
\nc{\Hzar}{\HHH_{\mathrm{Zar}}}
\nc{\Hdr}{\HHH_{\mathrm{dR}}}
\nc{\Het}{\HHH_{\mathrm{ét}}}
\nc{\Hcl}{\HHH_{\mathrm{cl}}}
\nc{\Hcet}{\HHH_{c,\mathrm{ét}}}

\nc{\rmT}{\mathrm{T}}
\nc{\Ts}{\rmT_{\sigma}}
\nc{\Tt}{\rmT_{\tau}}
\nc{\cXo}{\cX^{\circ}}
\nc{\rmW}{\mathrm{W}} 	

\title[Faisceau Automorphe Unipotent pour $\rmG_2$]{Faisceau Automorphe Unipotent pour $\rmG_2$, Nombres de Franel, et Stratification de Thom-Boardman}
\author{Lizao YE}
\date{3 février 2020}

\begin{document}

\maketitle

\sec*{Résumé}

Nous généralisons au cas équivariant un résultat de J. Denef et F. Loeser sur les sommes trigonométriques sur un tore; d'autre part, nous étudions la stratification de Thom-Boardman associée à la multiplication des sections globales des fibrés en droites sur une courbe. Nous montrons une inégalité subtile sur les dimensions de ces strates.

Notre motivation vient du programme de Langlands géométrique. En s'appuyant sur les travaux de W. T. Gan, N. Gurevich, D. Jiang et de S. Lysenko, nous proposons, pour le groupe réductif $G$ de type~$\rmG_2$, une construction conjecturale du faisceau automorphe dont le paramètre d'Arthur est unipotent et sous-régulier. En utilisant nos deux résultats ci-dessus, nous déterminons les rangs génériques de toutes les composantes isotypiques d'un faisceau $S_3$-équivariant qui apparaît dans notre conjecture, ce $S_3$ étant le centralisateur du $\SL_2$ sous-régulier dans le groupe dual de Langlands de $G$. 

\tableofcontents

\sec{Introduction}

	\ssec{Programme de Langlands géométrique}
	Une difficulté dans le programme de Langlands géométrique est de construire les faisceaux automorphes associés aux paramètres d'Arthur \emph{unipotents}.
	
	Soit $k$ un corps algébriquement clos de caractéristique~$p>0$. Soit $\ell$ un nombre premier différent de $p$. Nous employons la cohomologie étale $\ell$-adique. Soit $X$ une courbe projective lisse et irréductible sur~$k$, et $G$ un groupe réductif sur~$k$. Notons par $\Bun_G$ le champs classifiant des $G$-fibrés principaux sur $X$. Pour tout point $x\in X$, on a la grassmanienne affine $\Gr_x$ classifiant les $G$-fibrés principaux sur $X$ munis d'une trivialisation en dehors de $x$. On a 
	\[\Gr_x=G(K_x)/G(O_x).\]
	Ainsi il admet l'action du groupe $G(O_x)$ par multiplication à gauche.  La catégorie de Hecke $\Hk_x$ est la catégorie des faisceaux pervers $G(O_x)$-équivariants sur $\Gr_x$. Le produit tensoriel sur $\Hk_x$ est la convolution. On peut la munir naturellement d'une contrainte de commutativité. La catégorie monoïdale
			\[\Hk_X\colonequals\otimes '_{x\in X}\Hk_x\]
	agit, par modifications de $G$-fibrés, sur la catégorie 
	\[\Sh{\Bun_G}\]
	des faisceaux (dérivés) sur $\Bun_G$. 
	
	Le programme de Langlands géométrique vise à étudier la \emph{décomposition spectrale} de $\Sh{\Bun_G}$ sous cette action. On dit que $\cF\in\Sh{\Bun_G}$ est un \emph{faisceau propre pour Hecke} s'il est propre sous cette action, c'est-à-dire que pour tout $h\in\Hk_X$,
	\[h*\cF=\EF(h)\otimes\cF.\]
Ici $\EF(h)$ est un faiseau (dérivé) sur un point.

Quel système de valeurs propres $(\EF(h))_ {h\in\Hk_X}$ peut-on prendre? En fait, rappelons l'isomorphisme de Satake
\[\Sat_x:\Rep_{\check{G}}\ito\Hk_x,\]
où $\check{G}$ sur $\Qlb$ est le groupe réductif dual de $G$. 
Langlands et Arthur proposent que les $\EF(h)$ s'organisent en un $\check{G}$-système local \emph{gradué} $\sigma$ sur $X$, de sorte que pour toute représentation $V$ de $\check{G}$ et tout $x\in X$, notant $h=\Sat_x(V)$, on a un isomorphisme entre faisceaux dérivés sur un point:
\[V_{\sigma,x}=\EF(h).\]
           Ainsi,
                      
\begin{equation}\label{Hecke} 
          h*\cF=V_{\sigma,x}\otimes\cF.       
\end{equation}              

Un tel système local gradué $\sigma$ s'exprime aussi comme un homomorphisme
\[\sigma:\pi_1(X)\times\SL_2\to\check{G},\]
 appelé paramètre d'Arthur, où $\pi_1(X)$ désigne le groupe fondamental étale de $X$. En fait, pour toute représentation $V$ de $\check{G}$, $V_{\sigma}$ sera un système local sur $X$ muni d'une action de $\SL_2$. L'action du tore diagonal $\Gm\subset\SL_2$ fait de $V_{\sigma}$ un système local \emph{gradué}. Ainsi $V_{\sigma}$ peut être interprété comme \emph{dérivé}, c'est le sens adopté dans~\eqref{Hecke}.
 
 La direction \og galois$\implies$ automorphe\fg{} du programme de Langlands consiste, pour un paramètre d'Arthur $\sigma$ donné, à trouver tous les faisceaux $\cF$ propres pour Hecke pour le système local gradué associé à~$\sigma$.

\ssec{Conjecture d'Arthur}
Nous nous sommes concentrés sur les paramètres d'Arthur dits \emph{unipotents}, ce sont ceux qui se factorisent par $\SL_2\to\check{G}$. En termes de $\check{G}$-systèmes locaux sur $X$, ce sont des systèmes triviaux mais avec une graduation éventuellement non-triviale. Si un faisceau $\cF$ sur $\Bun_G$ correspond à un tel paramètre, l'action de Hecke sur $\cF$ devient très simple: pour tout $x\in X$ et toute représensation $V$ de $\check{G}$, notons $h\colonequals\Sat_x(V)$, alors 
\begin{equation}\label{Heckeunip}
            h*\cF=\VGm\otimes\cF.
\end{equation}            
 Ici, $\VGm$ est l'espace vectoriel $V$ muni de la graduation donnée par l'action de 
\[\Gm\subset\SL_2\to\check{G}.\]
Dans l'equation~\eqref{Heckeunip}, on considère $\VGm$ comme un faisceau \emph{dérivé} sur un point.  

Dans la suite, un paramètre d'Arthur sera sous-entendu unipotent, et sera noté simplement
\[\sigma:\SL_2\to\check{G}.\]
On dit que $\sigma$ est \emph{distingué} si son image n'est contenue dans aucun Levi propre de $\check{G}$.

Arthur, basé sur des considérations autour de la formule de trace, a conjecturé dans~\cite[Conjecture~8.1]{Arthur} l'existence et l'unicité de représentations automorphes unipotentes. L'analogue géométrique s'énonce comme suivant

\begin{Conj} \label{Conj:Arthur}
  Soit $\sigma$ un paramètre d'Arthur unipotent distingué, alors il existe un (essentiellement) unique faisceau (dérivé) $\cF$ sur $\Bun_G$ qui vérifie~\eqref{Heckeunip}.
  \end{Conj}

  L'unicité dans cet énoncé signifie que les autres faisceaux vérifiant les mêmes conditions sont des sommes directes de plusieurs copies (décalées) de $\cF$.
  
  Les classes de conjugaison d'homomorphisme $\sigma:\SL_2\to\check{G}$ sont en bijection avec les orbites unipotentes dans $\check{G}$. Cette bijection est donnée en prenant l'image par $\sigma$ d'un quelconque élément unipotent qui n'est pas l'identité de $\SL_2$. Les orbites unipotentes sont (partiellement) ordonnées:
  \[O\leq O' \text{ si }O\subset \bar{O'}.\]
  Il existe une unique orbite maximale, dite \emph{régulière}, et, si $\check{G}$ est simple, une unique orbite, dite \emph{sous-régulière}, qui est maximale parmi toutes les orbites non-régulières. Les paramètres d'Arthur correspondants se nomment de la même manière. 
    
  Pour un $\sigma$ régulier, le faisceau sur $\Bun_G$ qui lui correspond est le faisceau \emph{constant}. Si $G$ est de type~$\rm A$, le seul paramètre d'Arthur distingué est le paramètre régulier, donc la conjecture ne dit pas grand chose dans ce cas. Lorsque $G$ est de type~$\rm B, D, E, F, \text{ou }G$, mais non pas de type $\rm C$, l'orbite sous-régulière de $\check{G}$ est distinguée, voir~\cite{Collingwood}. Une question se pose donc naturellement dans ces cas:

  \begin{Q}\label{Question-sr}
     Comment décrire \emph{le} faisceau automorphe $\Fsr$ associé au paramètre d'Arthur unipotent sous-régulier?
     \end{Q}

Lorsque $G$ est de type $D$, cette question a été répondue dans~\cite{Lysenko} via la \og correspondence theta\fg{};\footnote{ La partie \og unicité\fg{} de la conjecture~\ref{Conj:Arthur} n'a néanmoins pas été considérée dans la littérature, et ne sera pas discutée dans cet article non plus.} lorsque $G$ est de type $E_6$ ou $E_7$, une construction conjecturale est proposée dans la conjecture 9.1 de \emph{loc. cit.} Remarquons que toutes ces constructions sont basées sur le fait que dans ces cas, le groupe $G$ admet au moins un sous-groupe parabolique maximal dont le radical unipotent est abélien. Ce n'est plus vrai pour $G$ de type $\rmG_2$.

 		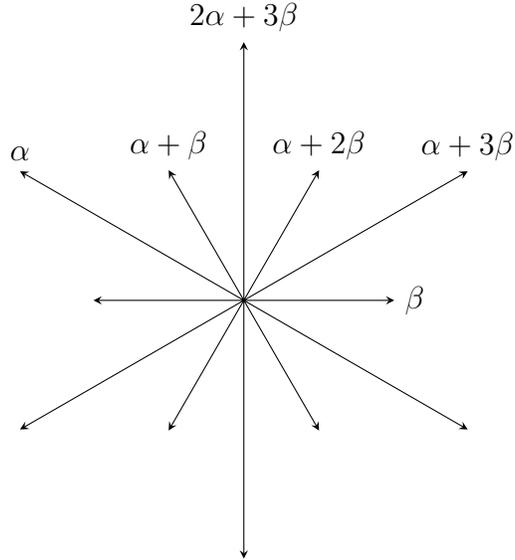
\begin{figure}[h]
		\begin{tikzpicture}[scale=2]
		\coordinate (0) at (0,0);
		\coordinate (1) at (0:1);
		\coordinate (2) at (30:1.71828);
		\coordinate (3) at (60:1);
		\coordinate (4) at (90:1.71828);
		\coordinate (5) at (120:1);
		\coordinate (6) at (150:1.71828);
		\coordinate (7) at (180:1);
		\coordinate (8) at (210:1.71828);
		\coordinate (9) at (240:1);
		\coordinate (10) at (270:1.71828);
		\coordinate (11) at (300:1);
		\coordinate (12) at (330:1.71828);		
		\foreach \x in {1,2,...,6} {\draw[<->,>=stealth] (60*\x-60:1)--(0)--(60*\x-30:1.71828);}				
		\draw (1) node[right]{$\b$};
		\draw (2) node[above]{$\a+3\b$};
		\draw (3) node[above]{$\a+2\b$};
		\draw (4) node[above]{$2\a+3\b$};		
		\draw (5) node[above]{$\a+\b$};
		\draw (6) node[above]{$\a$};
		\end{tikzpicture}
		\caption{Système des racines de $\rmG_2$}\label{Fig:systemeracines}
	\end{figure}

\ssec{Le cas du type $\rmG_2$ et notre conjecture}
En fait, le groupe $G$ de type $\rmG_2$ est le groupe de dimension la plus petite pour lequel la question~\ref{Question-sr} se pose. Dans ce cas, le faisceau $\Fsr$ est l'analogue géométrique de l'exemple le plus essentiel considéré par Langlands lui-même dans son livre~\cite{Langlands}. Il appraît dans l'Appendix III de \emph{loc. cit.}, comme illustration de sa méthode de prendre les résidus des \emph{séries d'Eisenstein} pour remplir le \emph{spectre discret}. Cet exemple a ensuite joué un rôle important quand Arthur formulait ses conjectures citées plus haut.\footnote{Il écrivait dans~\cite[p. 57]{Arthur}: \say{ It seems that the only other elliptic unipotent representation which is known to exist is the Langlands' representation for $\rmG_2$.}} Or, même si les séries d'Eisenstein géométriques existent, voir~\cite{Braverman}, jusqu'à présent on ne sait pas comment en prendre les \og\emph{résidus géométriques}\fg{}. 

Je propose une construction conjecturale de $\Fsr$ pour le groupe $G$ de type~$\rmG_2$. Pour l'énoncer, nous avons besoin de plus de notations.
 
	 Fixons un épinglage de $G$, c'est-à-dire un Borel $B$ de $G$ et un tore maximal $T$ dans $B$, d'où le système de racines. On note $\Phi$ l'ensemble des racines. Dans cet article, nous désignons par \og racines simples\fg{} celles qui sont habituellement désignées comme \og racines simples positives\fg{}. On a $\frg\colonequals\Lie G=\Lie(T)\oplus\oplus_{\gamma\in\Phi}\frg_{\g}$. On désigne par $\Ug$ le sous-groupe unipotent de $G$ d'algèbre de Lie $\frg_{\g}$. On note $\theta$ la plus haute racine. Les racines sont des $\ZZ$-combinaisons des racines simples. Si on fixe une racine simple et ne regarde que ses coefficients, on obtient une $\ZZ$-graduation sur $\frg$. Pour tout groupe réductif épinglé, on sait que ses paraboliques standards maximaux sont en bijection avec ses racines simples.  On note $P^{\g}$ le parabolique maximal correspondant à une racine simple $\g$. On note $M^{\g}$ le Levi standard de $P^{\g}$, et $U^{\g}$ le radical unipotent de $P^{\g}$, et $Z^{\g}$ le centre de $U^{\g}$.
	 
	 Dans le cas du type~$\rmG_2$, on note $\b$ la racine simple courte, qui induit une $\ZZ$-graduation \[\frg\colonequals\Lie G=\oplus_{n=-3}^3\frg_n.\]  Pour tout $n\geq 0$, on note $G_{\geq n}$ (resp. $G_0$) le sous-groupe de $G$ dont l'algèbre de Lie est $\oplus_{j\geq n}\frg_j$ (resp. $\frg_0$). Alors pour tout $m\geq n$, $G_{\geq m}$ est un sous-groupe normal de $G_{\geq n}$. On a \[ G_{\geq 0}=P^{\b}, G_0=M^{\b}, G_{\geq 1}=U^{\b}, G_{\geq 3}=Z^{\b}.\] On note $\a$ la racine simple longue. Alors \begin{gather*}\theta=2\a+3\b, \frg_3=\frg_{\theta}\oplus\frg_{\a+3\b},\\
	 \Lie(B/U_{\theta})=\Lie(TU_{\b})\oplus\frg_{\a+?\b},\end{gather*}
	 où $\frg_{\a+?\b}\colonequals\oplus_n\frg_{\a+n\b}$.	 	

	L'énoncé de notre construction se base sur le lemme suivant. 
	 \begin{Lm}
	     Supposons $\mathrm{car}(k)\neq 2\text{ ou }3$. Alors, il existe un unique morphisme entre des schémas \begin{equation}\label{Eq:Sthree}\frg_{\b}^3\times\frg_{\a+?\b}\to\frg_{\a+3\b}\end{equation} vérifiant les conditions suivantes: \begin{itemize} \item Il est $S_3$-invariant;\item Il est $TU_{\b}$-équivariant, où pour $t\in T$ et $x\in U_{\b}\ito\frg_{\b}$, l'action de $tx\in TU_{\b}$ sur $\frg_{\b}$ envoie $y\in\frg_{\b}$ sur $\Ad(t)(x+y)$; \item Sa restriction à $\{0\}^3\times\frg_{\a+3\b}$ est la projection;\item Sa restriction à $\frg_{\b}^3\times\frg_{\a}$ est non-nulle.\end{itemize} \end{Lm}

	 \begin{Rq}\label{Rq:formule} Exprimé par une formule, ce morphisme envoie $(x,y,z,\sigma_0+\cdots+\sigma_3)$, où $x,y,z\in\frg_{\b}, \sigma_n\in\frg_{\a+n\b}$, sur \[\sigma_3-\frac{x+y+z}{3}\sigma_2+\frac{xy+yz+zx}{6}\sigma_1-\frac{xyz}{6}\sigma_0.\] Ici $\frac{xy+yz+zx}{6}\sigma_1$ désigne $\frac{\ad(x)\ad(y)+\ad(y)\ad(z)+\ad(z)\ad(x)}{6}\sigma_1$, etc.
	 \end{Rq}

	La figure~\ref{Fig:complet} capture le diagramme des champs concernés.

	\begin{figure}
	\begin{tikzcd}
		\cY \ar[d, red, "\pi_{\cY}"] & \Bun_{B,\Om} \ar[l, swap,"j"] \ar[d, "\pi_B"] & & \Bun_{TU_{\b},\Om}^{sss} \ar[d, "\pi"] \\
		\Bun_{P^{\b}} \ar[d, "\pi_P"] & \Bun_{B/U_{\theta},\Om} \ar[dr, blue] & & \cY' \ar[dl, blue] \\
		\Bun_G & & \Bun_{TU_{\b},\Om} & 
	\end{tikzcd}
	\caption{} \label{Fig:complet}
	\end{figure}

Ici $\Om$ désigne le fibré en droites canonique sur la courbe $X$. S'il apparaît dans une indice, cela signifie des changements de base par rapport au morphisme $\Bun_{T,\Om}\to\Bun_T$ où $\Bun_{T,\Om}$ classifie  $\cF_T\in\Bun_T$ muni d'une \og trivialisation\fg{} $\ass{(\frg_{\a+3\b})}{T}\ito\Om$, où $\ass{(\frg_{\a+3\b})}{T}$ désigne le fibré en droites sur $X$ obtenu en tordant $\frg_{\a+3\b}$ par $\FT$ via l'action canonique de $T$ sur $\frg_{\a+3\b}$. Le champs $\cY$ classifie $\cF_{P^{\b}}\in\Bun_{P^{\b}}$ muni d'une application au-dessus de $X$ entre des fibrés vectoriels $\ass{(\frg_3)}{P^{\b}}\to\Om$. Le lieu ouvert de $\cY$ où cette dernière application est surjective est naturellement isomorphe à $\Bun_{B,\Om}$ via l'extension de groupes $B\to P^{\b}$. Les paires de flèches au-dessus d'une même base sont des fibrés vectoriels duaux. Donc $\cY'$ classifie $\cF_{TU_{\b}}\in\Bun_{TU_{\b},\Om}$ muni d'une application $\ass{(\frg_{\a+?\b})}{TU_{\b}}\to\Om$. Finalement $\Bun_{TU_{\b},\Om}^{sss}$ classifie 	$\cF_{TU_{\b}}\in\Bun_{TU_{\b},\Om}$ muni de \emph{trois} réductions numérotées dans $\Bun_{T,\Om}$. Il admet ainsi une action du groupe $S_3$. Le morphisme $\pi$ est obtenu en tordant le morphisme~\eqref{Eq:Sthree} par $\cF_{TU_{\b}}$ et en prenant ensuite des sections globales convenables. Il est fini et $S_3$-invariant. 

 Fixons tout au long de l'article un caractère nontrivial $\psi\colon\FF_p\to\Unit{\Qlb}$. Notons par $\Lpsi$ le \emph{faisceau d'Artin-Schreier} sur $\AA^1$ associé à $\psi$. Désignons par $\Four$ la transformation de Fourier identifiant la catégorie des faisceaux sur un fibré avec celle sur son dual via le noyau $\Lpsi$.

Une particularité: ayant la suite exacte non-scindée des homomorphismes de groupes
\[1\to Z^{\b}\to P^{\b}\to P^{\b}/Z^{\b}\to 1\] avec $Z^{\b}$ abélien, on a, pour un $\cF_{P^{\b}}$ donné, que la fibre du morphisme $\Bun_{P^{\b}}\to\Bun_{(P^{\b}/Z^{\b})}$ passant $\cF_{P^{\b}}$ est isomorphe à $\Bun_{(\ass{(Z^{\b})}{P^{\b}})}$. Ainsi la transformation de Fourier $\Fred$ associe à un faisceau sur $\Bun_{P^{\b}}$ un faisceau équivariant sur $\cY$. Son inverse $\Fred^{-1}=(\pi_{\cY})_!$.

Notre construction conjecturale est:
\begin{Conj}\label{Franel} 
Supposons $\mathrm{car}(k)\neq 2\text{ ou }3$. 
Alors, il existe un unique faiseau $\Fsr$ sur $\Bun_G$ tel que, sur chaque composante connexe $\Bun^d_{B,\Om}$ dont le degré $d=\dim\HHH^0(X, \ass{(U_{\b})}{T})$ est assez grand, on a (à décalage près) \[(\pi_P)^*\Fsr\ito (\pi_{\cY})_! j_{!*}(\pi_B)^*\Fblue(\pi_!\Ql).\]

Ce faisceau vérifie la propriété de Hecke~\eqref{Hecke} pour le paramètre d'Arthur unipotent sous-régulier~$\sigma:\SL_2\to\check{G}$.
    \end{Conj}
    
   Cette conjecture vient des calculs que nous avons effectués basés sur les articles~\cite{Lysenko} et~\cite{Gan}.

\sec{Résultats principaux}                  
   Comme première étape vers notre conjecture, nous avons mené une étude détaillée du faisceau $\Fblue(\pi_!\Ql)$ qui en est l'élément principal. Il s'agit donc de la figure~\ref{Fig:Four}, qui est une partie de la figure~\ref{Fig:complet}.

	\begin{figure}
	\begin{tikzcd}
		& & \Bun_{TU_{\b},\Om}^{sss} \ar[d, "\pi"] \\
		 \Bun_{B/U_{\theta},\Om} \ar[dr, blue] & & \cY' \ar[dl, blue] \\
		 & \Bun_{TU_{\b},\Om} & 
	\end{tikzcd}
	\caption{} \label{Fig:Four}
	\end{figure}

 Nous aboutissons en particulier au résultat suivant:
    
    \begin{Pp} \label{Pp:pervers}
    Supposons $\mathrm{car}(k)\neq 2\text{ ou }3$. Alors, pour $d$ assez grand, le faisceau $\Fblue(\pi_!\Ql)$ sur $\Bun^d_{B/U_{\theta},\Om}$ est pervers (décalé), génériquement un système local (décalé) de rang \[\rmF_d\colonequals\sum_{i=0}^d\binom{d}{i}^3.\]  
    
    Il admet une action du groupe $S_3$. Soit $\oplus_{\rho}\Frho\boxtimes\rho$ sa décomposition en composantes $S_3$-isotypiques, où $\rho$ parcourt les trois représentations irréductibles $\rhotriv,\rhost,\rhosgn$ de $S_3$. Alors $\Frho$ est un faisceau pervers irréductible dont la restriction sur un ouvert non-vide de $\Bun^d_{B/U_{\theta},\Om}$ est le décalage d'un système local de rang

\begingroup
\renewcommand{\arraystretch}{1.5}
\begin{tabular}{ll}
             $\rmF_d/6-(-2)^{d-1}+2^d/3$, & pour $\rho=\rhotriv$,\\
             $\rmF_d/3-2^d/3$, & pour $\rho=\rhost$,\\
             $\rmF_d/6+(-2)^{d-1}+2^d/3$, & pour $\rho=\rhosgn$.
\end{tabular}          
\endgroup

\end{Pp}

\begin{Rq} Cette $S_3$-symétrie doit venir du fait que le centralisateur de l'image du paramètre sous-régulier $\sigma: \SL_2\to\check{G}=\rmG_2$ est $S_3$.    \end{Rq}

\begin{Rq}
Ces nombres $\rmF_d$ sont appelés \emph{nombres de Franel}, puisque Franel, dans son article~\cite{Franel}, en répondant à une question de Laisant, a trouvé qu'ils satisfont
\[d^2\rmF_d = (7d^2 - 7d + 2)\rmF_{d-1}+ 8(d-1)^2\rmF_{d-2},\text{  pour tout }d\geq 2.\]

Pour $d=0,1,2,\ldots$, ces nombres $\rmF_d$ sont donnés par:

1, 2, 10, 56, 346, 2252, 15184, 104960, 739162, 5280932, 38165260, 278415920, 2046924400, 15148345760, 112738423360, 843126957056,~$\ldots$

Don Zagier l'appelle \og séquence A\fg{} dans~\cite{Zagier}, il l'a trouvée comme la première des six suites entières sporadiques satisfaisant une certaine équation de récurrence \og à la Apéry\fg{}. 

Dans un certain sens, nous introduisons une structure $S_3$-équivariante sur ces nombres.

On ne peut s'empêcher de spéculer sur la possibilité que ses autres suites sporadiques, qu'il a démontrées être \emph{modulaires}, soient liées aux autres faisceaux automorphes unipotents.
\end{Rq}

\begin{proof}[Démonstration de la proposition~\ref{Pp:pervers}] 

Suivons les notations de l'énoncé.

 Pour $d$ assez grand, le morphisme $\Bun_{T,\Om}\to\Bun_{TU_{\b},\Om}$ est un fibré vectoriel. En particulier, il est lisse et surjectif. Il suffit donc de démontrer l'analogue de l'énoncé après le changement de base par rapport à ce morphisme. Décrivons l'analogue de la figure~\ref{Fig:Four} après ce changement de base. Sur un point $\bullet\to\Bun_{T,\Om}$ correspondant à un fibré $\FT$ (muni de $\ass{(\frg_{\a+3\b})}{T}\ito\Om$) , en notant $\cL\colonequals\ass{(U_{\b})}{T}$ le fibré en droites associé, et d'après la remarque~\ref{Rq:formule} concernant le morphisme $\pi$, la figure~\ref{Fig:Four} devient la figure~\ref{Fig:Fourpt}.\footnote{A vrai dire, selon la remarque~\ref{Rq:formule},  le morphisme $\phi$ ici aurait dû être suivi de l'automorpisme linéaire de $\oplus_{n=0}^3\HX{\cL^{\otimes n}}$ induit par des homothéties de rapport $(-1)^n\frac{(3-n)!}{3!}$ sur $\HX{\cL^{\otimes n}}$ pour $n=0,1,\ldots,3$. Nous négligeons volontairement cet automorphisme, qui n'a pas d'impacts sur la suit de nos analyses, pour simplifier la présentation.}

	\begin{figure}
	\begin{tikzcd}
		& & \HX{\cL}^3 \ar[d, "\phi"] & (s_i)\ar[d, mapsto]\\
		\bigoplus\limits_{n=0}^3\HHH^1(X,\cL^{\otimes (-n)}\otimes\Omega) \ar[dr, blue] & &\bigoplus\limits_{n=0}^3\HX{\cL^{\otimes n}}\ar[dl, blue] & \prod(1+s_i)\\
		 &\bullet & &
	\end{tikzcd}
	\caption{} \label{Fig:Fourpt}
	\end{figure}

La perversité et l'irréductibilité des faisceaux $\cF_{\rho}$ se voient aussitôt car $\phi$, et donc $\pi$,  est fini et génériquement galoisien (sur son image) avec le groupe de Galois $S_3$. La partie concernant les rangs génériques des $\cF_{\rho}$ se reformule comme dans la proposition~\ref{Pp:traceIIthree} suivant.
\end{proof}

\begin{Pp} \label{Pp:traceIIthree}
Pour tout $\gamma\in S_3$ et tout $l\in\oplus_{n=0}^3\HHH^1(X,\cL^{\otimes (-n)}\otimes\Omega)$ générique,  \begin{multline*}\Tr(\gamma;\Hcet^\bullet(\HX{\cL}^3,(l\circ\phi)^*\Lpsi))\\=\begin{cases} (-2)^d, &\text{ lorsque }\gamma\text{ est transitif};\\2^d, &\text{ lorsque }\gamma \text{ est une transposition}; \\ (-1)^d\rmF_d, &\text{ lorsque }\gamma\text{ est l'identité}. \end{cases}\end{multline*}
\end{Pp}

Cette proposition sera le sujet du reste de l'article. Elle est un cas particulier de la proposition~\ref{Pp:TraceII}, compte tenu de la remarque~\ref{Rq:coeffs}. 

\ssec{Deux résultats en géométrie algébrique}
Notre démonstration de la proposition~\ref{Pp:traceIIthree} est basée sur deux résultats, qui nous paraissent nouveaux, en géométrie algébrique. Présentons-les dans cette partie.

  Le premier concerne les sommes trigonométriques sur les tores. Soit $k$ un corps algébriquement clos de caractéristique $p>0$. Soit $\rmT$ un tore sur $k$. Notons par $\Lambda$ le réseau des poids de $\rmT$. Soit $f=\sum_{\lambda\in\Lambda}c_{\lambda}t^\lambda$ une fonction sur $\rmT$. Notons par $\Delta(f)\subset\LtR$ l'enveloppe convexe de $\{\lambda|c_{\lambda}\neq 0\}$, et par $\Dinf(f)$ l'enveloppe convexe de $\{0\}\cup\Delta(f)$. On dit que $f$ est \emph{non-dégénérée à l'infini} si pour toute face $\Gamma$ de $\Dinf(f)$ ne contenant pas~$0$, les conditions suivantes sont satisfaites: \begin{itemize} \item Le diviseur sur $\rmT$ défini par $f|_{\Gamma}\colonequals\sum_{\lambda\in\Gamma\cap\Lambda}c_{\lambda}t^\lambda$ est lisse. \item Il existe $l\in\dual{\Lambda}$ telle que, considérée comme fonction sur $\Gamma$, elle est constante de valeur (entière) non-divisible par~$p$.\end{itemize}
     
     Un résultat de J. Denef et F. Loeser, voir~\cite[Theorem (1.3)]{Loeser}, dit que si $f$ est non-dégénérée à l'infini et si $\Delta\colonequals\Dinf(f)$ est de dimension égale à $\dim\rmT$, alors la caractéristique d'Euler-Poincaré de la cohomologie $\ell$-adique à support compact $\Hcet^\bullet(\rmT,f^*\Lpsi)$ est donnée par \[(-1)^{\dim\rmT}(\dim\Delta)!\Vol(\Delta).\]
     
     On le généralise au cas équivariant où $f$ est fixée par un groupe fini $\frG\subset\Aut(\Lambda)$. Pour tout $\gamma\in\frG$, notons par $\Delta^\gamma\subseteq\Delta$ le sous-polytope des points fixes sous $\gamma$. Alors, sous les mêmes hypothèses, nous montrons:
     \begin{Thm}[Théorème~\ref{Thm:trigtore} et Proposition~\ref{Pp:tracevol}] \label{Thm:principal} Pour tout $\gamma\in\frG$, \begin{multline*}\Tr(\gamma;\Hcet^\bullet(\rmT,f^*\Lpsi))\\=(-1)^{\dim\rmT}\det(\gamma;\Lambda)\det(1-\gamma;(\Delta^\gamma-\Delta^\gamma)^{\perp})\dim(\Delta^\gamma)!\Vol(\Delta^\gamma),\end{multline*} où $(\Delta^\gamma-\Delta^\gamma)^{\perp}$ désigne l'espace linéaire $\{l\in\dLtR|l\text{ est constante sur }\Delta^\gamma\}$.
     \end{Thm}
      Pour $\gamma=1$, on retrouve le résultat cité ci-dessus. 
     
     Pour montrer cette généralisation, on est réduit, en utilisant la \emph{formule de Grothendieck-Ogg-\v Safarevi\v c} comme dans~\cite[Proposition (3.2)]{Loeser},  à calculer \[\Tr(\gamma;\Hcet^\bullet(\rmT,G^{-1}(0)))\] où $G\colonequals f-c$ avec $c$ générique. L'analogue de cette trace en caractéristique~$0$ a été calculée par Stapledon dans~\cite[Theorem 6.5]{Stapledon}. Pour faire la comparaison, nous montrons qu'il existe un relèvement $\frG$-invariant $\hat{G}$ de $G$ sur l'\emph{anneau de Witt} $R\colonequals\mathrm{W}(k)$ et montrons que $\hat{G}^{-1}(0)$ admet une compactification lisse sur $R$ dans laquelle le complément de $\hat{G}^{-1}(0)$ est un diviseur strictement à croisements normaux relatif à $R$, et que l'action de $\frG$ sur $\hat{G}^{-1}(0)$ s'étend sur cette compactification. 
     
     En caractéristique~$0$, nous avons le résultat plus général suivant.
     \begin{Pp}[Propositions~\ref{Pp:Hilbert} et~\ref{Pp:coh_poly}] Soit $F$ un corps algébriquement clos de caractéristique~$0$, soit $\cP$ une variété lisse sur $F$, $\cX\subset\cP$ un diviseur lisse, et $D\hto\cP$ un diviseur strictement à croisements normaux tel que l'immersion fermée $D\cap\cX\hto\cX$ l'est aussi. Soit $\frG$ un groupe agissant sur $\cP$ de façon à préserver $\cX,D$. Supposons les conditions suivantes vérifiées:\begin{itemize}
    \item L'inclusion canonique $F\hto\HHH^0(\cP,\cO_{\cP})$ est un isomorphisme;
    \item Le fibré vectoriel $\Omega^1_{\cP}(\log D)$ sur $\cP$ est trivialisable;
    \item En tant que $F[\frG]$-module, $\HHH^0(\cP,\Omega^1_{\cP}(\log D))$ est isomorphe à son dual. On le note par $\rho$.
    \end{itemize}
Alors, il existe un (unique) polynôme $\phi(s)$ à coefficients dans le $\rmK$-groupe $\rmK(F[\frG]-\mathrm{mod})$, tel que la \emph{série de Hilbert-Poincaré} satisfait \[\sum_{n\geq 0}\HHH^\bullet(\cP,\cO(n\cX))s^n=\frac{\phi(s)}{(1-s)\det(1-s\rho)}\] dans $\rmK(F[\frG]-\mathrm{mod})[[s]]$, et tel que, en notant $\cXo\colonequals\cX-\cX\cap D$, la cohomologie de de Rham algébrique satisfait \[\Hdr^\bullet(\cXo)=\Hdr^\bullet(\cP-D)-(-1)^{\dim\cP}\det(\rho)\phi(1)\] dans $\rmK(F[\frG]-\mathrm{mod})$.
\end{Pp}

 Pour une fonction $G$ non-dégénérée et $\frG$-invariante sur un tore $\rmT$ comme ci-dessus, notons $\Delta=\Delta(G)$. On peut s'arranger pour trouver des triplets $(\cP,\cX,D)$ comme ci-dessus tels que \begin{gather*}\cXo=G^{-1}(0),\cP-D=\rmT,\rho=\Lambda\otimes F,\\\HHH^\bullet(\cP,\cO(n\cX))=F[(n\Delta)\cap\Lambda)].\end{gather*} C'est la base de tout nos calculs et fait le lien avec la \emph{série d'Ehrhart} \[\sum_{n\geq 0}F[(n\Delta)\cap\Lambda)]s^n\] dont la trace sous $\gamma\in\frG$, lorsque $s\to 1^-$, est asymptotiquement équivalente à \[\frac{\dim(\Delta^\gamma)!\Vol(\Delta^\gamma)}{(1-s)^{\dim(\Delta^\gamma)+1}}.\]

 \sssec{La non-dégénérescence}
  Pour analyser $\Fblue(\phi_!\Ql)$ ainsi que l'action de $S_3$ là-dessus, nous stratifions l'espace vectoriel $\oplus_{n=0}^3\HX{\cL^{\otimes n}}$ par des tores bien choisis, et on applique le théorème~\ref{Thm:principal} sur chaque tore. \footnote{Cette stratégie de stratifier un espace vectoriel par des tores se trouve aussi dans~\cite[Corollary 0.10]{Fu}. La différence est que, d'une part, son article ne considère pas les cas équivariants, c'est-à-dire avec l'action d'un groupe comme notre $S_3$ ici. D'autre part, ses espaces vectoriels sont \og sans structures internes\fg{}. Les nôtres sont formés des sections globales des fibrés en droites sur une courbe, le choix d'une base bien adaptée n'est donc plus automatique.}
 Ainsi, il reste à montrer la condition de non-dégénérescence requise dans~\ref{Thm:principal}. Cette condition se trouve être hautement non-triviale dans notre situation. On la démontre finalement sous une forme plus générale en utilisant la \emph{stratification de Thom-Boardman} qui à un morphisme $\pi\colon\cX\to\cY$ entre des variétés lisses associe la stratification de $\cX$ par les strates\[\Sigma^r(\pi)\colonequals\{x\in\cX|\dim\diff{\pi}\T_x\cX=r\},\] où $r$ parcourt les entiers naturels. 

Voici le deuxième résultat en géométrie algébrique mentionné plus haut.

  \begin{Thm}[Théorème~\ref{Thm:ThomII}]
     Soit $X/k$ une courbe projective, lisse, et géométriquement irréductible. Soit $I$ un ensemble fini. Soit $\cL$ un fibré en droites sur $X$. Soient $(\cL_i)_{i\in I}$ des sous-faisceaux de $\cL$. Prenons la $k$-algèbre intègre $\prod_{n\geq 0}\HX{\cL^{\otimes n}}$ et notons par $\cK$ son corps de fraction. Pour tout $i\in I$, prenons un $k$-sous-espace vectoriel $V_i$ de $\cK$ qui est soit $k\oplus\HX{\cL_i}$, soit  $\HX{\cL_i}$. Alors le morphisme \og produit\fg{}
      \[m:\prod_{i\in I}\PP(V_i)\to\PP(\cK)\]
                                   satisfait 
\[\dim\Sigma^r(m)\leq r, \text{ pour tout } r\geq 0.\]
    \end{Thm}

  \textbf{Remerciements.} Cet article est la thèse de l'auteur, faite à l'Institut \'Elie Cartan de  l'Université de Lorraine. L'auteur remercie son directeur de thèse S. Lysenko pour son encouragement et pour de nombreuses discussions.

\sec{Stratification de Thom-Boardman}
Soit $k$ un corps algébriquement clos de caractéristique quelconque. Dans ce paragraphe, toutes les variétés sont définies sur~$k$, et un point d'une telle variété signifie un $k$-point. Pour un point $z$ sur une telle variété, on désigne par $\cO_z$ l'anneau local en $z$ et $\frm_z$ l'ideal maximal de $\cO_z$.

Soit $\phi:\cX\to\cY$ un morphisme entre des variétés lisses. Il induit, sur tout point $x$ de $\cX$, l'application tangente
\[\diff{\phi}:\T_x\cX\to\T_{\phi(x)}\cY.\]
Généralisant les travaux de Morse sur les \og singularités \fg{} de $\phi$, Thom, dans~\cite{Thom}, a associé à $\phi$ une stratification de $\cX$ par des sous-variétés 
\[\Sigma^r(\phi)\colonequals\{x\in\cX|\dim\diff{\phi}\T_x\cX=r\}, r\geq 0.\]
Cette stratification a ensuite été poursuivie par Boardman dans~\cite{Boardman}.

\ssec{Une condition sur les dimensions des strates}\label{ssec:Thomineq}
Il est dans notre intérêt de connaître les dimensions des $\Sigma^r(\phi)$. On a par exemple la formule \og produit des corangs \fg{} si $\phi$ est\og générique\fg{}, voir~\cite[Théorème 2]{Thom}. Attention néanmoins à la différence entre nos notations.

 On se demande en particulier si la condition suivante est vérifiée:
\begin{equation}\label{Eq:dimension}\dim\Sigma^r(\phi)\leq r, \text{ pour tout } r\geq 0.
\end{equation}
Elle est équivalente à: il existe une famille finie de variétés $\cX_i$ au-dessus de $\cX$ telles que les images des $\cX_i$ couvrent $\cX$ et pour tout $x$ dans l'image de $\cX_i$, on a
         \[\dim\diff{\phi}\T_x\cX\geq\dim\cX_i.\]
         Dans cet article, on montre toujours cette condition équivalente pour justifier~\eqref{Eq:dimension}.
         
       L'intérêt principal selon nous de la condition~\eqref{Eq:dimension} est le suivant:

       \begin{Lm}\label{Lm:seclisse}
             Soit $\cY,\cZ$ des variétés lisses. Soit $\cL$ un fibré en droites sur $\cZ$ tel que $\HHH^0(\cZ,\cL)$ est de dimension finie sur~$k$ et que pour tout point $z\in\cZ$, l'application canonique $\HHH^0(\cZ,\cL)\to\cL\otimes\cO_z/\frm_z^2$ est surjective.  Soit $\phi:\cY\to\cZ$ un morphisme vérifiant~\eqref{Eq:dimension}. Alors toute section globale $s\in\HHH^0(\cZ,\cL)$ générique définit un diviseur lisse sur $\cY$.
       \end{Lm}

       \begin{proof}
           Pour toute section $s\in\HHH^0(\cZ,\cL)$ et tout point $y\in\cY$, notons $\cS_y$ le noyau de l'application canonique $\HHH^0(\cZ,\cL)\to\cL\otimes\cO_y/\frm_y^2$. Alors, $s$ définit un diviseur lisse sur $\cY$ au voisinage de $y$ si et seulement si $s\notin\cS_y$. Ainsi, $s$ définit un diviseur lisse sur $\cY$ si et seulement si $s\notin\bigcup_{y\in\cY}\cS_y$.
           
           D'autre part, pour tout point $y\in\cY$, notons $z\colonequals\phi(y)$, alors les conditions suivantes sont équivalentes:\begin{itemize} \item $y\in\Sigma^r(\phi)$. \item Le noyau de l'application canonique $\cO_z/\frm_z^2\to\cO_y/\frm_y^2$ est de codimension $r+1$. \end{itemize}De plus, comme $\HHH^0(\cZ,\cL)\to\cL\otimes\cO_z/\frm_z^2$ est surjective, elles sont encore équivalentes à \begin{itemize} \item $\cS_y$ est de codimension $r+1$ dans $\HHH^0(\cZ,\cL)$. \end{itemize} Ainsi la codimension de $\bigcup_{y\in\Sigma^r(\phi)}\cS_y$ dans $\HHH^0(\cZ,\cL)$ est au moins $r+1-\dim\Sigma^r(\phi)$. Ayant supposé $\dim\Sigma^r(\phi)\leq r$, cette codimension est au moins~$1$. Il s'ensuit que la codimension de $\bigcup_{y\in\cY}\cS_y=\bigcup_r\bigcup_{y\in\Sigma^r(\phi)}\cS_y$  dans $\HHH^0(\cZ,\cL)$ est aussi au moins~$1$, ce qui termine la preuve.                                 
       \end{proof}

       On en utilisera plutôt le corollaire suivant:
       \begin{Lm} \label{Lm:duallisse}
              Soit $V$ un espace vectoriel. Soient $\cX, \cY$ des variétés lisses. Soit   \[\begin{array}{ccc} 
           \cX &\overset{\phi_1}{\to} & V-\{0\}\\
                      \pi\da &   & \da    \\
           \cY &\overset{\phi}{\to}   &\PP(V)
           \end{array}\] un diagramme commutatif où $\pi$ est lisse. Supposons $\phi$ satisfait~\eqref{Eq:dimension}, alors pour tout $l\in\dual{V}$ générique, la fonction $l\circ\phi_1$ définit un diviseur lisse sur $\cX$.                     
        \end{Lm}

        Si $V$ est de dimension infinie, on le considère comme l'ind-schéma $``\colim_U\text{''} U$ où $U$ parcourt les sous-espaces vectoriels de $V$ de dimensions finies. Similairement pour $V-\{0\}$ et $\PP(V)$.

        \begin{proof}
             Prenons $\cZ\colonequals\PP(V)$. La projection canonique $V-\{0\}\to\PP(V)$ est un $\Gm$-torseur sous l'action canonique de $\Gm$ sur $V-\{0\}$ par dilatation, et donc correspond à un fibré en droites sur $\cZ$. Prenons $\cL$ le fibré en droites dual. Alors, le diagramme commutatif dans l'énoncé n'est rien d'autre qu'une trivialisation du tiré en arrière de $\cL$ sur $\cX$. On a canoniquement $\HHH^0(\cZ,\cL)\ito\dual{V}$. Ce que l'on veut montrer devient alors: toute $s\in\HHH^0(\cZ,\cL)$ générique définit un diviseur lisse sur $\cX$. 
             
             Comme $\pi$ est lisse, il suffit de montrer que: toute $s\in\HHH^0(\cZ,\cL)$ générique définit un diviseur lisse sur $\cY$. Ceci découle du lemme~\ref{Lm:seclisse} appliqué au morphisme $\phi$, car pour tout $z\in\cZ$, l'application canonique $\HHH^0(\cZ,\cL)\to\cL\otimes\cO_z/\frm_z^2$ est bijective, donc est en particulier surjective.
        \end{proof}

         \ssec{Multiplication des espaces vectoriels \emph{dans} un corps}\label{Sec:corps}
         Pour tout espace vectoriel $V$ sur $k$, on considère $\PP(V)$ comme l'ind-schéma $``\colim_U\text{''} \PP(U)$ où $U$ parcourt les sous-$k$-espaces vectoriels de $V$ de dimensions finies. Alors, pour toute $k$-droite $L\subset V$, on a canoniquement \[\T_L\PP(V)\ito\Hom_k(L,V)/k.\]
         
         Soit $A$ une $k$-algèbre (commutative). Notons $\Unit{A}$ le groupe des éléments inversibles de $A$. Alors, les faisceaux en groupes pour la topologie \emph{fpqc}, $\Res_{A|k}\Gm$ et $\Res_{A|k}(\Gm)/\Gm$, sont naturellement des ind-schémas agissant sur l'ind-schéma $\PP(A)$. En effet, sur un $k$-schéma $\cS$, ces trois faisceaux classifient respectivement: \begin{itemize}\item  les sections globales sur $\cS$ du faisceau $\Unit{(A\otimes_k\cO_{\cS})}$. \item les classes d'isomorphisme des triplets $(\cL,s_1,s_2)$ où $\cL$ est un fibré en droites sur $\cS$, et $s$ (resp. $s'$) est une section globale sur $\cS$ du faisceau $A\otimes_k\cL$ (resp. $A\otimes_k\dual{\cL}$) telles que $s\otimes_A s'=1$ en tant que section globale du faisceau $A\otimes_k\cO_{\cS}$. \item les classes d'isomorphisme des couples $(\cM,s_0)$ où $\cM$ est un fibré en droites sur $\cS$, et $s_0$ est une section globale sur $\cS$ du faisceau $A\otimes_k\cM$ telle que \[\Hom_k(A,k)\otimes_k\cO_{\cS}\to\cM, e\otimes f\mapsto f\cdot\coupl{e\otimes\id}{s_0}\] est un épimorphisme. \end{itemize} Sous ces termes, l'action de $\Res_{A|k}(\Gm)/\Gm$ sur $\PP(A)$ est donnée par: \[(\cL,s_1,s_2)\cdot(\cM,s_0)\colonequals(\cL\otimes_{\cO_{\cS}}\cM, s_1\otimes_A s_0).\] Au niveau des $k$-points, ces actions deviennent les actions naturelles de $\Unit{A}$ et de $\Unit{A}/\Unit{k}$ sur $(A-\{0\})/\Unit{k}$ par multiplication. On remarque aussi que $\Res_{A|k}(\Gm)/\Gm$, en oubliant $s_2$, est naturellement un sous-faisceau de $\PP(A)$.

         Si $A$ est lui-même un corps $\cK$, alors, $\Unit{\cK}=\cK-\{0\}$. 
         \begin{Pp}\label{Pp:schfini} L'inclusion d'ind-schémas \begin{equation}\label{Eq:inclu}\Res_{\cK|k}(\Gm)/\Gm\subseteq\PP(\cK)\end{equation} induit des bijections au niveau de leurs valeurs sur tout $k$-schéma \emph{fini}. 
         
         De plus, l'action ci-dessus de $\Res_{\cK|k}(\Gm)/\Gm$ sur $\PP(\cK)$ devient \emph{transitive} et \emph{libre} au niveau de leurs valeurs sur tout $k$-schéma \emph{fini}.
         \end{Pp}
          Ainsi les espaces tangents sur tous les $k$-points de $\PP(\cK)$ sont tous \emph{canoniquement} isomorphes à \[\Lie(\Res_{\cK|k}(\Gm)/\Gm)\ito\cK/k.\]          
         
          D'ailleurs, l'ensemble des $k$-sous-espaces vectoriels de $\cK$ forment un \emph{demi-anneau} sous l'addition usuelle et la multiplication donnée par: $V\cdot W$ est le $k$-sous-espace vectoriel de $\cK$ engendré par les $vw$ où $v\in V, w\in W$. Les éléments multiplicativement inversibles sont exactement les $k$-droites dans $\cK$. Ainsi pour tout $k$-sous-espace vectoriel $V\subset\cK$ et toute $k$-droite $L\subset\cK$, le quotient $\frac{V}{L}$ est bien défini est désigne l'unique sous $k$-space vectoriel de $\cK$ dont le produit avec $L$ donne $V$. Ceci induit une multiplication $m$ sur l'ind-schéma $\PP(\cK)$, compatible avec celle sur le groupe $\Res_{\cK|k}(\Gm)/\Gm$ et l'inclusion~\eqref{Eq:inclu}.

On a le diagramme commutatif suivant pour tout $k$-sous-espace vectoriel $V$ de $\cK$ et toute $k$-droite $L\subset V$: 
   \[\begin{array}{ccc} 
           \T_L\PP(V) & \hookrightarrow & \T_L\PP(\cK)\\
           \|                                            &                                                        & \|      \\
           \frac{V}{L}/k                                &\hookrightarrow           &\cK/k
           \end{array}
\]
En outre, soient $L_1,L_2\subset\cK$ des $k$-droites, et notons $L\colonequals L_1\cdot L_2\subset\cK$ le produit. Alors on a diagramme commutatif
\[\begin{array}{ccc} 
           \prod_{i=1}^2\T_{L_i}\PP(\cK) & \overset{\diff{m}}{\longrightarrow} & \T_L\PP(\cK)\\
           \|                                            &                                                        & \|      \\
           \prod_{i=1}^2(\cK/k)&\overset{+}{\longrightarrow}           &\cK/k
           \end{array}
\]
Pour le voir, il suffit, grâce à la proposition~\ref{Pp:schfini}, de le comparer avec le diagramme correspondant pour la multiplication sur le \emph{groupe} $\Res_{A|k}(\Gm)/\Gm$.

   Dans ce cadre, une question naturelle se pose:
 \begin{Q} Soit $I$ un ensemble fini. Soient $(V_i)_{i\in I}$ des $k$-sous-espaces vectoriels de dimension finie de $\cK$. Est-ce que le morphisme \og produit\fg{}
                         \[m:\prod_{i\in I}\PP(V_i)\to\PP(\cK)\]
                 satisfait~\eqref{Eq:dimension}?
 \end{Q}
 
On n'en traite ici que deux cas particuliers. La réponse est affirmative dans ces deux cas. 

 \sssec{Cas \textrm{I}}
\begin{Pp} \label{Pp:ThomI}
   Soit $E$ un espace affine sur $k$. Soit $\cK\colonequals k(E)$ le corps des fonctions rationnelles sur $E$. Soient $V_i\subset\cK$ des $k$-sous-espaces vectoriels de dimension finie, tous constitués d'applications affines sur $E$. Alors le morphisme \og produit\fg{}
                                   \[m:\prod_{i\in I}\PP(V_i)\to\PP(\cK)\]
                                   satisfait~\eqref{Eq:dimension}.
\end{Pp}
\begin{proof}
Considérons les classes d'isomorphisme de la donnée de \begin{itemize}\item un ensemble $J_0$ avec un élément marqué $*\in J_0$;\item une application $\pi: I\to J_0$\end{itemize} sujette aux conditions suivantes:\begin{itemize}\item $k\subset V_i$ si $\pi(i)=*$;\item notons $J\colonequals J_0-\{*\}$, alors aucun des $V_j\colonequals\cap_{i,\pi(i)=j}V_i$ pour $j\in J$ n'est contenu dans $k$, et que parmi eux, ceux qui sont de dimension 1 sont deux-à-deux différents. \end{itemize} 

Pour une telle classe, prenons la variété $\cX_{J_0, *, \pi}$, notée plus simplement par $\cX_{\pi}$,  des  $(L_i) \in\prod_{i\in I}\PP(V_i)$ satisfaisant \begin{itemize}\item $L_i=k$ si et seulement si $\pi(i)=*$; \item $L_i=L_{i'}$ si et seulement si $\pi(i)=\pi(i')$.\end{itemize} 
  Ces variétés, non-vides, forment une stratification de $\prod_{i\in I}\PP(V_i)$. En outre, on a une immersion ouverte \[\cX_{\pi}\subset\prod_{j\in J} \PP(V_j).\]  Il s'ensuit que \[\dim\cX_{\pi}=\sum_{j\in J}\dim\PP(V_j).\] On notera $(L_j)\in\prod_{j\in J} \PP(V_j)$ l'image de $(L_i)$ sous cette immersion.
  
  Pour conclure, il suffit de montrer que pour tout $(L_i)\in\cX_{\pi}$, on a \[\dim\diff{m}\T_{(L_i)}\prod_{i\in I}\PP(V_i)\geq\dim\cX_{\pi}.\] Suivant le formalisme du paragraphe~\ref{Sec:corps}, on a \[\diff{m}\T_{(L_i)}\prod_{i\in I}\PP(V_i)=(\sum_i\frac{V_i}{L_i})/k=(\sum_{i,\pi(i)=*}V_i+\sum_{j\in J}\frac{\sum_{i,\pi(i)=j}V_i}{L_j})/k.\]
Ainsi, \[\diff{m}\T_{(L_i)}\prod_{i\in I}\PP(V_i)\supset(\sum_{j\in J}\frac{V_j}{L_j})/k.\]
Il reste à montrer que \[\dim ((\sum_{j\in J}\frac{V_j}{L_j})/k))\geq\sum_{j\in J}\dim\PP(V_j),\]
lequel découle du lemme suivant.
\end{proof}                                   
\begin{Lm}
    Soit $E$ un espace affine sur $k$. Soient  $(l_j)_j, (v_j)_j$ des applications affines sur $E$, où $j$ parcourt un ensemble fini non-vide $J$. Supposons les $l_j$ deux-à-deux non-proportionnelles et qu'aucune des $l_j$ et des $\frac{v_j}{l_j}$ n'est constante. Alors $\sum_j\frac{v_j}{l_j}$ n'est pas constante non plus.
\end{Lm}
\begin{proof}
    C'est clair car la fonction $\sum_j\frac{v_j}{l_j}$ sur $E$ a des singularités sur tous les hyperplans définis par $l_j=0$.
\end{proof}

\sssec{Cas \textrm{II}}
  Pour donner un aperçu de ce dont il s'agit, on l'énonce d'abord dans une situation légèrement plus simple.
   \begin{Thm}\label{Pp:simple}
         Soit $X/k$ une courbe projective, lisse, et géométriquement irréductible. Soit $\cK\colonequals k(X)$ le corps des fonctions rationnellles sur $X$. Soit $D_i\hookrightarrow X$ des diviseurs, et \[V_i\colonequals\{f\in\cK|\mathrm{div}(f)\geq -D_i\}\subset\cK.\] Alors le morphisme \og produit\fg{}
                                   \[m:\prod_{i\in I}\PP(V_i)\to\PP(\cK)\]
                                   satisfait~\eqref{Eq:dimension}.
 \end{Thm}
En fait, plus généralement, 
 \begin{Thm}\label{Thm:ThomII}
     Soit $X/k$ une courbe projective, lisse, et géométriquement irréductible. Soit $I$ un ensemble fini. Soit $\cL$ un fibré en droites sur $X$. Soient $(\cL_i)_{i\in I}$ des sous-faisceaux de $\cL$. Prenons la $k$-algèbre intègre $\prod_{n\geq 0}\HX{\cL^{\otimes n}}$ et notons par $\cK$ son corps de fraction. Pour tout $i\in I$, prenons un $k$-sous-espace vectoriel $V_i$ de $\cK$ qui est soit $k\oplus\HX{\cL_i}$, soit  $\HX{\cL_i}$. Alors le morphisme \og produit\fg{}
      \[m:\prod_{i\in I}\PP(V_i)\to\PP(\cK)\]
                                   satisfait~\eqref{Eq:dimension}.
    \end{Thm}
    Pour retrouver la proposition~\ref{Pp:simple}, prenons, pour tout $i$, $\cL_i=\cO(D_i)$ pour un diviseur $D_i\hookrightarrow X$ et $V_i\colonequals\HX{\cL_i}$. Alors, le corps $\cK$ défini ici sera différent de celui pris dans la proposition~\ref{Pp:simple}. Cette différence n'a pas d'impact sur l'énoncé, et n'est en fait nécessaire que si certain $V_i=k\oplus\HX{\cL_i}$. 
    \begin{proof}[Démonstration de~\ref{Thm:ThomII}]
     On peut supposer que les $\cL_i$ sont tous non-nuls. Ils sont alors aussi des fibrés en droites.
     
        Considérons les classes d'isomorphisme de la donnée de \begin{itemize}\item un ensemble $J$ muni d'une une partition $\alpha: J=J_1\sqcup J_2$;\item une application surjective $\pi:I\to J$;\item un entier naturel $m$\end{itemize} sujette aux conditions suivantes:     
         \begin{itemize}\item $\pi(i)\in J_2$ si $V_i=\HX{\cL_i}$;\item l'application restreinte $\pi^{-1}(J_2)\overset{\pi}{\to} J_2$ est bijective.\end{itemize} 
        
        Pour une telle classe, notons $\cL_j\colonequals\cap_{i,\pi(i)=j}\cL_i$ pour tout $j\in J$, et prenons la variété $\cX_{J,\alpha,\pi,m}$, notée plus simplement $\cX_{\pi,m}$, qui classifie la donnée de\begin{itemize}\item pour tout $j\in J_2$, une suite décroissante des fibrés en droites sur $X$ \[\cL_j=\cL_{j,0}\supseteq \cL_{j,1}\supseteq\cdots\supseteq \cL_{j,m};\]\item pour tout $j\in J_1$, une section $s_j\in\HX{\cL_j}$\end{itemize} sujette aux conditions suivantes:
        \begin{itemize}\item les $(s_j)_{j\in J_1}$ sont deux-à-deux différentes;\item pour tout $n\in\{1,\ldots,m\}$, au moins un $j\in J_2$ est tel que \[\HX{\cL_{j,n-1}}\neq\HX{\cL_{j,n}};\]\item pour tout $n\in\{1,\ldots,m\}$, il existe un point $x_n\in X$, nécessairement unique, tel que pour tout $j\in J_2$, l'inclusion $\cL_{j,n-1}\subset \cL_{j,n}$ est un isomorphisme en dehors de $x_n$;\item les $(x_n)$ sont deux-à-deux différents; \item pour tout $j\in J_2$, $\HX{\cL_{j,m}}$ est de dimension 1, et de plus, si $\cO_X\subset \cL_{j,m}$ représente un élément non-nul dedans, alors pour tout $n\in\{1,\ldots,m\}$, l'inclusion induite $\cO_X\subset \cL_{j,n}$ est un isomorphisme en $x_n$.\end{itemize}
        
        Remarquons tout de suite que $m$ ne peut pas être trop grand pour que cette variété soit non-vide, ainsi seul un nombre fini de ses variétés sont non-vides. D'autre part, comme $s_j\in\HX{\cL_j}$ pour tout $j\in J_1$ et que $(x_1,\ldots, x_m)\in X^m$, on a un morphisme naturel \[\cX_{\pi,m}\to\prod_{j\in J_1}\HX{L_j}\times X^m.\] Ce morphisme est \emph{quasi-fini}, ce qui donne l'estimation \[\dim\cX_{\pi,m}\leq\sum_{j\in J_1}\dim\HX{\cL_j}+m.\]
        
        D'autre côté, on a un morphisme naturel \[\cX_{\pi,m}\to\prod_{i\in I}\PP(V_i)\] envoyant une donnée comme ci-dessus à la donnée, pour tout $i\in I$, de la droite $L_i\subset V_i$ définie, en notant $j\colonequals\pi(i)$, comme $k\cdot(1+s_j)$ si $j\in J_1$, et comme $\HX{\cL_{j,m}}$ si $j\in J_2$. On observe que: 
        
        {\centering \emph{les images de ces morphismes couvrent $\prod_{i\in I}\PP(V_i)$.}
        
        }
         Ainsi, pour conclure, il suffit de montrer que  \[\dim\diff{m}\T_{(L_i)}\prod_{i\in I}\PP(V_i)\geq\dim\cX_{\pi,m}.\] Montrons, plus fortement, que \[\dim\diff{m}\T_{(L_i)}\prod_{i\in I}\PP(V_i)\geq\sum_{j\in J_1}\dim\HX{\cL_j}+m.\]  Suivant le formalisme du paragraphe~\ref{Sec:corps}, on a \[\diff{m}\T_{(L_i)}\prod_{i\in I}\PP(V_i)=(\sum_i\frac{V_i}{L_i})/k,\] et donc est égal à \[(\sum_{j\in J_1}\frac{\sum_{i,\pi(i)=j}V_i}{1+s_j}+\sum_{j\in J_2}\frac{\sum_{i,\pi(i)=j}V_i}{L_j})/k.\]
        Il s'ensuit que $\diff{m}\T_{(L_i)}\prod_{i\in I}\PP(V_i)$ contient les deux sous-espaces vectoriels\begin{itemize}\item $\sum_{j\in J_1}\frac{\HX{\cL_j}}{1+s_j}$;\item $(\sum_{j\in J_2}\frac{\HX{\cL_j}}{L_j})/k$.\end{itemize} Pour conclure, on est réduit à montrer les faits suivants\begin{itemize}\item a) L'intersection de ces deux sous-espaces est 0. \item b) La somme $\sum_{j\in J_1}\frac{\HX{\cL_j}}{1+s_j}$ est une somme directe.\item c) $\dim((\sum_{j\in J_2}\frac{\HX{\cL_j}}{L_j})/k)\geq m$. \end{itemize}
        
        Voici les démonstrations de ces faits.
        
        Pour a). Remarquons d'abord que l'anneau $\prod_{n\geq 0}\HX{\cL^{\otimes n}}$ est le compété de $\oplus_{n\geq 0}\HX{\cL^{\otimes n}}$ qui est un anneau $\NN$-gradué. En utilisant le développement dans $\cK$: \[\frac{1}{1+s_j}=1-s_j+s_j^2-\cdots,\] on voit que tout élément dans $\frac{\HX{\cL_j}}{1+s_j}$ est la somme (infinie) des éléments de degré strictement positif. Par contre, chaque $\frac{\HX{\cL_j}}{L_j}$ est composé d'éléments de degré 0. 
        
        Pour b). Pour tout $j\in J_1$, soit $v_j\in\HX{\cL_j}$. On a \[\sum_{j\in J_1}\frac{v_j}{1+s_j}=\sum_{n\geq 0}(-1)^n\sum_{j\in J_1}v_js_j^n.\] Donc $\sum_{j\in J_1}\frac{v_j}{1+s_j}=0$ si et seulement si pour tout $n\geq 0$, \[\sum_{j\in J_1}v_js_j^n=0.\] On conclut en remarquant que la matrice \[(s_j^n)_{j\in J_1, 0\leq n<|J_1|}\] est inversible, les $(s_j)_{j\in J_1}$ étant deux-à-deux différentes.   
        
         Pour c). C'est la partie la plus essentielle. Soit $n\in\{1,\ldots,m\}$. Alors, pour tout $j\in J_2$, tout élément de $\frac{\HX{\cL_{j,n-1}}}{L_j}$ est \emph{régulier} sur $x_1,\ldots, x_{n-1}$; D'autre part, d'après la définition de $\cX_{\pi,m}$, on peut prendre un $j(n)\in J_2$ tel que $\HX{\cL_{j(n),n-1}}\neq\HX{\cL_{j(n),n}}$, alors, un élément de $\frac{\HX{\cL_{j(n),n-1}}}{L_{j(n)}}$ est régulier sur $x_n$ si et seulement s'il est dans $\frac{\HX{\cL_{j(n),n}}}{L_{j(n)}}$. 
         
         Donc pour $n=1,\ldots, m,$ et pour tout $f_n$ dans l'ensemble non-vide \[\frac{\HX{\cL_{j(n),n-1}}}{L_{j(n)}}\setminus\frac{\HX{\cL_{j(n),n}}}{L_{j(n)}},\] on aura $f_1,\ldots, f_m$ et $1$ sont $k$-linéairement indépendants.
    \end{proof}

\sec{Caractéristique d'Euler-Poincaré équivariante}\label{Sec:Euler}
\ssec{Complexe de de Rham logarithmique}\label{Sec:de_Rham}
  Soit $F$ un corps algébriquement clos de caractéristique 0.    
  
   Soit $\cP$ une variété lisse sur $F$. Soit $\cX\hto\cP$ une hypersurface fermée et lisse. Soit $D\hto\cP$ un diviseur à croisements normaux stricts, tel que $D\cap\cX\hto\cX$ est un diviseur à croisements normaux stricts de $\cX$. Notons $j:\cXo\to\cX$ l'ouvert complémentaire au diviseur $D\cap\cX$. 
      
   Notons par $\Aut(\cP,\cX, D)$ le groupe des automorphismes de $\cP$ préservant respectivement $\cX$ et $D$. On va exprimer la cohomologie de $\cXo$, ou plutôt son image dans le $\rmK$-\emph{groupe de Grothendieck} des $F[\Aut(\cP,\cX,D)]$-modules, en termes des cohomologies des faisceaux des formes différentielles sur $\cP$ ayant des singularités \og contrôlées\fg{} le long de $X$ et de $D$. On note ce $\rmK$-groupe $\rmK(F[\Aut(\cP,\cX,D)]-\mathrm{mod})$.
   
   Par \og forme différentielle sur $\cP$ à singularités \emph{logarithmiques} le long de $D$\fg{}, on entend une forme différentielle $\omega$ sur l'ouvert $\cP-D$ vérifiant: $\omega$ et $\rmd\omega$ ont tous les deux des pôles d'ordre au plus 1 le long de $D$. Prenons $\Omega_{\cP}^\bullet(\log D)$ le complexe des faisceaux cohérents de telles formes différentielles sur $\cP$ à singularités \emph{logarithmiques} le long de $D$. Similairement, on prend le complexe des faisceaux $\Omega_{\cX}^\bullet(\log D\cap\cX)$ sur $\cX$. Prenons $\cL\colonequals\cO(X)$ le fibré en droites sur $\cP$. Notons, pour tout $n,m\in\ZZ$, \[\Omega^{n,m}\colonequals\Omega_{\cP}^m(\log D)\otimes\cL^{\otimes n}.\] On a des inclusions \[\cdots\subseteq\Omega^{n-1,m}\subseteq\Omega^{n,m}\subseteq\cdots.\]

   \begin{Lm}
      On a canoniquement des suites exactes longues \[\cdots\to\Omega^{n,m}/\Omega^{n-1,m}\to\Omega^{n+1,m+1}/\Omega^{n,m+1}\to\cdots.\]
   \end{Lm}

   \begin{proof}
      Localement sur $\cP$, l'idéal définissant $\cX$ est principal, donc est engendré par une fonction $f$. Les flèches dans les suites en question sont données par \[\cdot\wedge\frac{\rmd f}{f}.\] Ces flèches ainsi définies sont indépendantes du choix de $f$. L'exactitude de ces suites est claire.
   \end{proof}

   Ayant en même temps les suites exactes \[\Omega^{-1,m-1}/\Omega^{-2,m-1}\to\Omega^{0,m}/\Omega^{-1,m}\to\Omega^m_{\cX}(\log D\cap\cX)\to 0,\] il s'ensuit des suites exactes longues 
   \begin{equation}\label{Eq:Omega_resolution}
         0\to\Omega^m_{\cX}(\log D\cap\cX)\to\Omega^{1,m+1}/\Omega^{0,m+1}\to\Omega^{2,m+2}/\Omega^{1,m+2}\to\cdots.\end{equation}

   On cite aussi le lemme suivant, dont la version analytique est le célèbre lemme d'Atiyah-Hodge.
   \begin{Lm}\label{qiso}
    L'inclusion \[\Omega^\bullet_{\cX}(\log D\cap\cX)\to j_*\Omega_{\cXo}^\bullet\] est un quasi-isomorphisme.
   \end{Lm}

   \begin{proof}
      L'assertion étant locale, on peut supposer $\cX$ affine. Ensuite, un dévissage standard nous ramène au cas où $F=\CC$ est le corps des nombres complexes. Dans ce cas, ce lemme est conséquence des théorèmes 4.4 et 4.6 de~\cite{Atiyah}. 
   \end{proof}

   \begin{Pp}\label{Pp:cohomology_normal_crossing}
      La cohomologie de de Rham algébrique $\Hdr^\bullet(\cXo)$ est égale à \[\Hdr^\bullet(\cP-D)-\sum_m(-1)^m\HHH^\bullet(\cP,\Omega^{m,m})\] dans $\rmK(F[\Aut(\cP,\cX,D)]-\mathrm{mod})$.
   \end{Pp}

   \begin{proof}   
      Calculons dans le $\rmK$-groupe $\rmK(F[\Aut(\cP,\cX,D)]-\mathrm{mod})$.
       
          La cohomologie de de Rham algébrique  $\Hdr^\bullet(\cXo)$ est définie comme l'\emph{hypercohomologie} \[\HH^\bullet(\cXo,\Omega_{\cXo}^\bullet)\] du complexe $\Omega_{\cXo}^\bullet$ sur $\cXo$. Comme $j$ est affine, le foncteur \emph{image directe} $j_*$ sur les faisceaux abéliens est exact lorsqu'on le restreint aux faisceaux quasi-cohérents. En particulier, on a un quasi-isomorphisme \[j_*\Omega_{\cXo}^\bullet=\rmR j_*\Omega_{\cXo}^\bullet.\] Ainsi, l'hypercohomologie ci-dessus est isomorphe à 
              \[\HH^\bullet(\cX, j_*\Omega_{\cXo}^\bullet).\] Compte tenu du lemme~\ref{qiso}, on a \[\Hdr^\bullet(\cXo)\ito\HH^\bullet(\cX,\Omega^\bullet_{\cX}(\log D\cap\cX)).\]
        
        D'autre part, par~\eqref{Eq:Omega_resolution}, on a pour tout $j\in\ZZ$, \begin{multline}\label{Eq:Omega_resolution_cohomology}(-1)^j\HHH^\bullet(\cX,\Omega^j_{\cX}(\log D\cap\cX))=\sum_{n\geq 0, m=n+j+1}(-1)^m\HHH^\bullet(\cP,\Omega^{n,m})\\-\sum_{n>0, m=n+j}(-1)^m\HHH^\bullet(\cP,\Omega^{n,m}).\end{multline}      
        Sommons par rapport aux $j\geq 0$, on obtient
         \begin{align*}\Hdr^\bullet(\cXo)&=\sum_{m>n\geq 0}(-1)^m\HHH^\bullet(\cP,\Omega^{n,m})-\sum_{m\geq n>0}(-1)^m\HHH^\bullet(\cP,\Omega^{n,m})\\&=\sum_m(-1)^m\HHH^\bullet(\cP,\Omega^{0,m})-\sum_m(-1)^m\HHH^\bullet(\cP,\Omega^{m,m})\\&=\Hdr^\bullet(\cP-D)-\sum_m(-1)^m\HHH^\bullet(\cP,\Omega^{m,m}). \end{align*}       
   \end{proof}
 
   \ssec{Théorie d'Ehrhart équivariante}
   
   Reprenons les notations du paragraphe~\ref{Sec:de_Rham}. Soit $\frG\hto\Aut(\cP,\cX,D)$ un sous-groupe fixe.

    \begin{Lm}\label{Lm:independance_difference} Les sommes \[\sum_{n,m\geq 0,n-m=i}(-1)^m\HHH^\bullet(\cP,\Omega^{n,m}) \] pour $i\geq 0$ sont toutes égales.\end{Lm}

    \begin{proof}    Il suffit de prendre $j=-1,-2,\ldots$ dans~\eqref{Eq:Omega_resolution_cohomology}.
   \end{proof}

   Pour faire le lien avec la \emph{théorie d'Ehrhart}, il nous faut l'hypothèse:
   \begin{Hyp}\label{Hyp:Omega_trivial}
          \hfill
    \begin{enumerate} 
    \item \label{Item:Omega_trivial_functions}L'inclusion canonique $F\hto\HHH^0(\cP,\cO_{\cP})$ est un isomorphisme.
    \item  \label{Item:Omega_trivial_trivialisable}Le fibré vectoriel $\Omega^1_{\cP}(\log D)$ sur $\cP$ est trivialisable.
    \item \label{Item:Omega_trivial_duality} En tant que $F[\frG]$-module, $\HHH^0(\cP,\Omega^1_{\cP}(\log D))$ est isomorphe à son dual. On le note par $\rho$.
    \end{enumerate}
    \end{Hyp}

    \begin{Pp}\label{Pp:Hilbert}
              Sous l'hypothèse~\ref{Hyp:Omega_trivial}, il existe un (unique) polynôme $\phi(s)$ à coefficients dans le $\rmK$-groupe $\rmK(F[\frG]-\mathrm{mod})$, tel que la \emph{série de Hilbert-Poincaré} \[\sum_{n\geq 0}\HHH^\bullet(\cP,\cL^{\otimes n})s^n=\frac{\phi(s)}{(1-s)\det(1-s\rho)}\]
    \end{Pp}

    \begin{proof}
        Notons $d=\dim\cP$.
        
         Par les hypothèses~\ref{Item:Omega_trivial_functions} et~\ref{Item:Omega_trivial_trivialisable}, l'application canonique \[\cO_{\cP}\otimes\rho\to\Omega^1_{\cP}(\log D)\] est un isomorphisme. Ainsi, pour tout $m\geq 0$, \[\Omega^{n,m}\ito\cL^{\otimes n}\otimes\wedge^m\rho.\] Donc \[\HHH^\bullet(\cP,\Omega^{n,m})\ito\HHH^\bullet(\cP,\cL^{\otimes n})\otimes\wedge^m\rho.\] 
         
On a $\wedge^m\rho\simeq\det(\rho)\otimes\wedge^{d-m}(\dual{\rho})$ pour $ m=0,1,\ldots, d$. Par l'hypothèse~\ref{Item:Omega_trivial_duality}, $\dual{\rho}\simeq\rho$, donc \[\wedge^m\rho\simeq\det(\rho)\otimes\wedge^{d-m}\rho.\] Il s'ensuit que 
\begin{equation}\label{Eq:Omega_duality}
\HHH^\bullet(\cP,\Omega^{n,m})\simeq\det(\rho)\otimes\HHH^\bullet(\cP,\Omega^{n,d-m}).
\end{equation} 
Comme \[\det(1-s\rho)=1-s\rho+s^2\wedge^2\rho-\cdots,\]
le produit  \[\det(1-s\rho)\sum_{n\geq 0}\HHH^\bullet(\cP,\cL^{\otimes n})s^n\] se développe donc comme \[\sum_{n, m\geq 0}(-1)^m\HHH^\bullet(\cP,\cL^{\otimes n})\wedge^m(\rho) s^{n+m},\] qui se note aussi comme  \[\sum_{n,m\geq 0}(-1)^m\HHH^\bullet(\cP,\Omega^{n,m}) s^{n+m}.\] Par~\eqref{Eq:Omega_duality}, c'est égal à  \[\sum_{n\geq 0, 0\leq m\leq d}(-1)^m\det(\rho)\HHH^\bullet(\cP,\Omega^{n,d-m}) s^{n+m},\] ou bien, changeant l'indice $m$ en $d-m$, à la somme de  \[(-1)^d\det(\rho)\sum_{m>n\geq 0}(-1)^m\HHH^\bullet(\cP,\Omega^{n,m}) s^{n-m+d}\] et de \[(-1)^ds^d\det(\rho)\sum_{n\geq m}(-1)^m\HHH^\bullet(\cP,\Omega^{n,m}) s^{n-m}.\] On observe, grâce à~\eqref{Lm:independance_difference}, que 
 \[\sum_{n\geq m}(-1)^m\HHH^\bullet(\cP,\Omega^{n,m}) s^{n-m}=\frac{1}{1-s}\sum_m(-1)^m\HHH^\bullet(\cP,\Omega^{m,m}).\]
Ainsi le polynôme \begin{multline}\phi(s)\colonequals(-1)^d(1-s)\det(\rho)\sum_{m>n\geq 0}(-1)^m\HHH^\bullet(\cP,\Omega^{n,m}) s^{n-m+d}\\+(-1)^ds^d\det(\rho)\sum_m(-1)^m\HHH^\bullet(\cP,\Omega^{m,m})\end{multline} est celui requis par l'énoncé de la proposition.
    \end{proof}

La fin de cette preuve nous donne:
\begin{Cor}
   Le polynôme $\phi(s)$ de la proposition précédente vérifie \[\phi(1)=(-1)^{\dim\cP}\det(\rho)\sum_m(-1)^m\HHH^\bullet(\cP,\Omega^{m,m}).\]
\end{Cor}

Ce corollaire nous permet, sous les mêmes hypothèses, de réécrire la proposition~\ref{Pp:cohomology_normal_crossing}:
\begin{Pp}\label{Pp:coh_poly}
Sous l'hypothèse~\ref{Hyp:Omega_trivial}, on a  \[\Hdr^\bullet(\cXo)=\Hdr^\bullet(\cP-D)-(-1)^{\dim\cP}\det(\rho)\phi(1).\]
\end{Pp}

 \ssec{Caractéristique d'Euler-Poincaré équivariante des diviseurs sur un tore}
  Soit $R$ un anneau de valuation discrète d'idéal maximal $\frm$ et de corps résiduel $k$.  Tout schéma dans cette partie est défini sur $R$. Pour nous, l'avantage principal de cette restriction sur $R$ réside dans la facilité de reconnaître les modules plats sur $R$: ce sont juste les modules \emph{sans torsion}. On utilisera à plusieurs reprises, mais souvent implicitement, le lemme suivant:

       \begin{Lm}\label{Lm:flat} Soit $R$ un anneau de valuation discrète d'idéal maximal $\frm$. Soit $A$ un anneau \emph{factoriel} contenant $R$ comme sous-anneau. Supposons $A/\frm A$ intègre. Soit $a \in A$ non-nul, alors, $a$ définit un diviseur sur $\Spec(A)$ relatif à $R$, i.e. $A/a$ est plat sur $R$, si et seulement si $a\not\in\frm A$.
               \end{Lm}

               \begin{proof} Soit $\pi\in\frm$ un générateur. Comme $A/\frm A$ est intègre, $\pi$ est \emph{premier} dans $A$. Ainsi, \begin{align*} A/a \text{ est plat sur } R &\Equiv A/a\overset{\cdot\pi}{\to}A/a\text{ est injective}\\&\Equiv (aA)\cap(\pi A)=a\pi A\\&\Equiv a\not\in\pi A=\frm A.\end{align*} \end{proof}

  Le cas dégénéré où $R$ est un corps n'est pas exclu.
  
    Soit $\Lambda$ un réseau, i.e. un groupe abélien libre et de rang fini. 
  
  Pour tout anneau $A$, on désigne par $\rmT_A$ le tore $\Spec(A[\Lambda])$, et par $\rho_A$ le $A$-module $\Lambda\otimes A$. Notons $\rmT=\rmT_{R}$ sauf mention contraire explicite.

   \sssec{Diviseurs sur un tore}
    Une fonction sur $\rmT$ s'écrit comme une somme finie \[f=\sum_{\lambda\in\Lambda}c_{\lambda}t^{\lambda}\] où les coefficients $c_{\lambda}\in R$. Le \emph{support} de $f$ est défini comme \[\supp(f)\colonequals\{\lambda|c_{\lambda}\neq 0\}.\]  Pour tout sous-ensemble $S\subseteq\LtR$, on note \[f|_S\colonequals\sum_{\lambda\in S\cap\Lambda}c_{\lambda}t^{\lambda}.\]  On notera $\Delta(\supp(f))$ aussi par $\Delta(f)$, appelé \emph{polytope de Newton} associé à $f$, où $\Delta(S)$ désigne l'enveloppe convexe de $S$. On note $f_k$ l'image de $f$ dans $k[\Lambda]$. On a \[\Delta(f_k)\subseteq\Delta(f).\]
  
   L'anneau $R[\Lambda]$ étant \emph{principal}, tout diviseur sur $\rmT$ est \emph{principal}, i.e. est globalement défini par une fonction sur $\rmT$. Cette fonction est alors déterminée à multiplication par des éléments du groupe $\Unit{R[\Lambda]}=\Unit{R}t^{\Lambda}$ des fonctions inversibles près.

  \sssec{Variétés toriques}
  On suit les notations du livre~\cite{Fulton} en ce qui concerne les variétés toriques. 

   Si $R$ est un corps, rappelons qu'une variété torique sous le tore $\rmT$ est une variété $\cP$, supposée normale, contenant $\rmT$ comme un ouvert dense telle que l'action par multiplication de $\rmT$ sur $\rmT$ s'étend, de manière unique bien sûr, à une action de $\rmT$ sur $\cP$. Les variétés toriques sous le tore $\rmT$ sont \emph{en bijection} avec les \emph{éventails} dans $\dLtR$: la variété torique $\cP$ associée à un éventail $\frF$ dans $\dLtR$ est construite dans la catégorie des schémas comme le \emph{recollement} 
     \[\cP\colonequals\colim_{\sigma\in\frF}U_{\sigma}\] 
     des variétés toriques affines \[U_{\sigma}\colonequals\Spec(R[\dual{\sigma}\cap\Lambda])\] où \[\dual{\sigma}\colonequals\{\lambda\in\LtR|\coupl{\lambda}{\sigma}\geq 0\}.\] Les $U_{\sigma}$ deviennent alors des ouverts affines $\rmT$-stables de $\cP$.
    
    Cette construction reste valable pour $R$ général.

      Dans $U_{\sigma}$, il y a une unique $\rmT$-orbite fermée. Elle est donnée par \[\Ts\colonequals\Spec(R[\sigma^{\perp}\cap\Lambda])\hto U_{\sigma}\] où $\sigma^{\perp}$, noté habituellement $\mathrm{cospan}(\dual{\sigma})$, est défini comme 
      \[\{\lambda\in\LtR|\coupl{\lambda}{\sigma}=0\}\subseteq\dual{\sigma}.\] Remarquons que les orbites $\Ts$ sont naturellement des \emph{tores}, et $\rmT$ agit sur $\Ts$ à travers l'homomorphisme canonique \[\rmT\twoheadrightarrow \Ts.\] 
       Ces homomorphismes admettent des sections homomorphes, et sont, en particulier, lisses.
       
       La collection $(\Ts)_{\sigma\in\frF}$ \emph{est} la stratification de $\cP$ par ses $\rmT$-orbites. On a \begin{gather*}U_{\sigma}=\bigcup_{\tau\in\frF,\tau\subseteq\sigma}\Tt\\ \overline{\Ts}=\bigcup_{\tau\in\frF, \tau\supseteq\sigma} \Tt.\end{gather*}           
   
    Pour tout $\sigma,\tau\in\frF$, $\overline{\Ts}$ est couvert par les ouverts $U_{\tau}$ où $\tau$ parcourt les cônes dans $\frF$ contenant $\sigma$. On a $\overline{\Ts}\cap U_{\sigma}=\Ts$, et \[\overline{\Ts}\cap U_{\tau}=\begin{cases}\Spec (R[\dual{\tau}\cap\sigma^{\perp}\cap\Lambda]),\text{ si }\sigma\subseteq\tau,\\\emptyset,\text{ sinon}.\end{cases}\]
   
   La variété torique $\cP$ est lisse (sur $R$) si et seulement si le éventail $\frF$ est \emph{non-singulier}, i.e. tout cône dans $\frF$ est engendré par une partie d'une $\ZZ$-base de $\dual{\Lambda}$; $\cP$ est propre (sur $R$) si et seulement si $\frF$ est \emph{complet}, i.e. son \emph{support} $|\frF|\colonequals\bigcup_{\sigma\in\frF}\sigma$ est égal à $\dLtR$.
   
   Notons l'\emph{intérieur} d'un cône $\sigma$ par $\ocirc{\sigma}$. C'est le complément de l'union de toutes les faces \emph{propres} de $\sigma$.

   \sssec{\'Eventail associé à un polytope} 
   Tout polytope $\Delta\subset\LtR$ crée une relation d'équivalence sur $\dLtR$: deux éléments de $\dLtR$ sont équivalents s'ils \emph{déterminent} la même face de $\Delta$. Les adhérences de ces classes d'équivalences forment alors un éventail \emph{complet}, noté $\frF(\Delta)$. \footnote{Si $\Delta$ n'était pas de dimension \emph{maximale}, $\frF(\Delta)$ dans $\dLtR$ serait l'\og image inverse\fg{} d'un éventail dans $(\dLtR)/((\Delta-\Delta))^{\perp})$.} Si le polytope est rationel, le éventail le sera aussi.

 \begin{Lm}\label{Lm:polyfaces}
         Soit $\Delta$ un polytope dans $\LtR$. Soit $\frF$ un éventail raffinant $\frF(\Delta)$. Alors toute face de $\Delta$ est de la forme $\Gamma(\sigma)$ pour un (ou plusieurs) $\sigma\in\frF$.
 \end{Lm}

   \sssec{Compactification d'un diviseur sur un tore}\label{Sec:comp}
   Soit $\cP$ une variété torique sur $R$ sous le tore $\rmT$. Supposons $\cP$ lisse sur $R$, construite à partir d'un éventail non-singulier $\frF$. Notons $\frF(1)$ le sous-ensemble des cônes de dimension 1 dans $\frF$. On sait que:

   \begin{Lm}\label{Lm:partial} Le sous-schéma fermé $D\colonequals\cP-\rmT\hto\cP$ est un diviseur à croisements normaux stricts relatifs à $R$, dont les composantes irréductibles sont les $\overline{\Tt}$, où $\tau$ parcourt $\frF(1)$.
   
    Pour tout $I\subseteq\frF(1)$, l'intersection $\bigcap_{\tau\in I}\overline{\Tt}$ est non-vide si et seulement si les éléments de $I$ engendrent un cône $\sigma\in\frF$. Dans ce cas, cette intersection est égale à $\overline{\Ts}$.
    \end{Lm}

   Soit $\cXo\hto \rmT$ un diviseur. Soit $\cX\hto\cP$ l'\emph{adhérence schématique} de $\cXo$. On a bien $\cXo=\cX-\cX\cap D$.
   
   Fixons une fonction $f$ définissant le diviseur $\cXo\hto \rmT$. Le choix de $f$ est seulement pour faciliter l'articulation des énoncés suivants dont le contenu ne dépend pas vraiment de ce choix. 
   
   Le éventail $\frF(\Delta(f))$ associé au polytope $\Delta(f)$ ne dépendant pas du choix de $f$, on le note aussi par $\frF(\cXo)$.

   Le lemme suivant est clair:   
   \begin{Lm}\label{Lm:sigma} Soit $\sigma\in\frF$. Alors:
     
      Le sous-schéma fermé $\cX\cap U_{\sigma}\hto U_{\sigma}$ est défini par $t^{-\lambda_{f,\sigma}}f$ pour les $\lambda_{f,\sigma}\in\Lambda$ tels que $\Delta(f)-\lambda_{f,\sigma}$ est contenu dans $\dual{\sigma}$ et touche $\sigma^{\perp}$.  
     
     De tels $\lambda_{f,\sigma}$ forment une classe modulo $\sigma^{\perp}\cap\Lambda$. Ainsi, \begin{gather*}A_{f,\sigma}\colonequals\sigma^{\perp}+\lambda_{f,\sigma},\\B_{f,\sigma}\colonequals\dual{\sigma}+\lambda_{f,\sigma}\end{gather*} ne dépendent pas du choix de $\lambda_{f,\sigma}$. 
     
     Pour tout $\tau\in\frF$ contenu dans $\sigma$, le sous-schémas fermé \[\cX\cap\overline{\Tt}\cap U_{\sigma}\hto\overline{\Tt}\cap U_{\sigma}\] est défini par \[(t^{-\lambda_{f,\sigma}}f)|_{\tau^{\perp}}.\]     
     
     Pour tout $l\in\ocirc{\sigma}$, le minimum de $l$ sur $\Delta(f)$ est atteint précisément sur \[A_{f,\sigma}\cap\Delta(f).\] Cette intersection est donc une \emph{face} de $\Delta(f)$. On la note par $\Gamma(\sigma)$.

     Le sous-schéma fermé $\cX\cap\Ts\hto\Ts$ est défini par $t^{-\lambda_{f,\sigma}}(f|_{\Gamma(\sigma)})$.
       \end{Lm}
       
       On en déduit:
        \begin{Lm}\label{Lm:flatorb}       Soit $\sigma\in\frF$. 
        
               Alors les conditions suivantes sont équivalentes:
                              \begin{itemize}
               \item $\cX\cap\Ts$ est plat sur $R$.              
               \item L'anneau $R[\sigma^{\perp}\cap\Lambda]/f|_{\Gamma(\sigma)}$ est plat sur $R$.
               \item $f|_{\Gamma(\sigma)}\notin\frm[\Lambda]$.
               \item Le sous-schéma fermé de $\rmT$ défini par $f|_{\Gamma(\sigma)}$ est plat sur $R$, i.e. l'anneau $R[\Lambda]/f|_{\Gamma(\sigma)}$ est plat sur $R$. 
                \item Pour tout $\tau\in\frF$, $\cX\cap\overline{\Tt}\cap U_{\sigma}$ est plat sur $R$. 
                               \end{itemize}                                                                                                  
          \end{Lm}

             \begin{Lm}\label{Lm:smoothorb}       Soit $\sigma\in\frF$. 
        
               Alors, les conditions suivantes sont équivalentes:
                              \begin{itemize}
               \item $\cX\cap\Ts$ est lisse sur $R$.              
               \item Le sous-schéma fermé de $\rmT$ défini par $f|_{\Gamma(\sigma)}$ est lisse sur $R$.
                               \end{itemize}                                                                                                  
          \end{Lm}

        \begin{Lm}  \label{Lm:flatallorb}  
       Les conditions suivantes sont équivalentes:\begin{itemize}
               \item Pour tout $\sigma\in\frF$, $\cX\cap\Ts$ est plat sur $R$.              
               \item Pour tout $\sigma\in\frF$, $\cX\cap\overline{\Ts}$ est plat sur $R$. 
                               \end{itemize}                                                                                             
         De plus, si $\frF$ raffine $\frF(\cXo)$, ces conditions sont équivalentes à:
                           \[\Delta(f_k)=\Delta(f).\]
        \end{Lm}

        \begin{proof} On a les équivalences suivantes:
                         \begin{align*}&\text{Pour tout }\sigma\in\frF, \cX\cap\overline{\Ts}\text{ est plat sur }R\\
                                         \Equiv &\text{Pour tout }\sigma,\tau\in\frF, \cX\cap\overline{\Ts}\cap U_{\tau}\text{ est plat sur }R\\
                                        \Equiv &\text{Pour tout }\tau\in\frF, \cX\cap\Tt\text{ est plat sur }R,
                           \end{align*}
                        dont la deuxième vient du lemme~\ref{Lm:sigma}.
                        
                         La dernière assertion du lemme est conséquence des lemmes~\ref{Lm:polyfaces} et~\ref{Lm:flatorb}.
         \end{proof}

   \begin{Lm}\label{Lm:F_non_degenerate}Les conditions suivantes sont équivalentes:
      \begin{enumerate}
   \item \label{Itm:normal} $\cX$ est lisse sur $R$ et $D\cap\cX\hto\cX$ est un diviseur à croisements normaux stricts relatif à $R$.
   \item \label{Itm:orbitscl} Pour tout $\sigma\in\frF$, $\cX\cap\overline{\Ts}$ est lisse sur $R$.
   \item \label{Itm:orbits} Pour tout $\sigma\in\frF$, $\cX\cap\Ts$ est lisse sur $R$.
   \item \label{Itm:faces} Pour tout $\sigma\in\frF$, le sous-schéma fermé de $\rmT$ défini par $f|_{\Gamma(\sigma)}$ est lisse sur $R$.
    \end{enumerate}
   \end{Lm}

   \begin{proof}
               L'équivalence \ref{Itm:normal}$\Equiv$~\ref{Itm:orbitscl} se déduit du lemme~\ref{Lm:partial}.
               
              L'implication \ref{Itm:orbitscl}$\implies$~\ref{Itm:orbits} est automatique puisque $\Ts$ est un ouvert dans $\overline{\Ts}$.
                            
               Montrons \ref{Itm:orbits}$\implies$~\ref{Itm:orbitscl}. Supposons~\ref{Itm:orbits}. Pour tout $\sigma\in\frF$, le schéma $\overline{\Ts}$ est lisse sur $R$ et est stratifié par des strates $(\Tt)_{\tau\in\frF,\tau\supseteq\sigma}$ lisses sur $R$. De plus, le diviseur $\cX\cap\overline{\Ts}\hto\overline{\Ts}$ est plat sur $R$ par lemme~\ref{Lm:flatallorb}, et son intersection avec toute strate $\Tt$ reste lisse sur $R$ par~\ref{Itm:orbits}. On conclut par lemme suivant. 
                           
               Pour l'équivalence \ref{Itm:orbits}$\Equiv$~\ref{Itm:faces}, on se réfère au lemme~\ref{Lm:smoothorb}.
   \end{proof}

   \begin{Lm}
          Soit $\cY\to\cZ$ un morphisme lisse entre des schémas. Soit $\cX\hto\cY$ un diviseur plat sur $\cZ$. Supposons que $\cY$ est union des sous-schémas localement fermés $(\cY_i)$ lisses sur $\cZ$, tels que pour tout $i$, $\cX\cap\cY_i$ est un diviseur sur $\cY_i$ et est lisse sur $\cZ$. Alors $\cX$ est lui aussi lisse sur $\cZ$.
   \end{Lm}

   \begin{proof}
          Comme $\cX$ est plat sur $\cZ$, il est lisse sur $\cZ$ si toutes les fibres géométriques sont lisses. Ainsi on peut supposer que $\cZ=\Spec(\kappa)$ où $\kappa$ est un corps algébriquement clos. Le problème étant Zariski-local par rapport à $\cY$, on peut supposer que le diviseur $\cX\hto\cY$ est défini par une fonction $f$ sur $\cY$. Alors pour tout $i$, $\cX\cap\cY_i$ est le sous-schéma fermé de $\cY_i$ défini par $f$. Il faut montrer que pour tout $\kappa$-point $x\in\cX$, $\rmd f$ n'est pas identiquement nul sur l'espace tangent $\T_x\cY$. 
          
          Comme les $(\cY_i)$ couvrent $\cY$, $x$ appartient à un certain $\cY_{i_0}$, et donc à $\cX\cap\cY_{i_0}$, lequel est lisse par notre hypothèse. On en déduit que $\rmd f$ n'est pas identiquement nul sur l'espace tangent $\T_x\cY_{i_0}$ qui est un sous-espace de $\T_x\cY$.
   \end{proof}

\begin{Def}
Le diviseur $\cXo\hto \rmT$ est dit $\frF$-\emph{non-dégénéré} s'il vérifie les conditions équivalentes du lemme~\ref{Lm:F_non_degenerate}.
\end{Def}

\begin{Def} 
 Le diviseur $\cXo\hto \rmT$ est dit \emph{non-dégénéré} si pour toute face $\Gamma$ de $\Delta(f)$, le sous-schéma fermé de $\rmT$ défini par $f|_{\Gamma}$ est lisse sur $R$.
 
 Une fonction $f\in R[\Lambda]$ est dite \emph{non-dégénérée} si le sous-schéma fermé du $\rmT$ défini par $f$ est un diviseur non-dégénéré.
 \end{Def}

 \begin{Lm}\label{Lm:equivnondeg} Soit $\cXo\hto\rmT$ un diviseur. Alors les conditons suivantes sont équivalentes:
         \begin{itemize}
         \item
      Il est non-dégénéré.
      \item Il est $\frF$-non-dégénéré pour tout éventail $\frF$ non-singulier dans $\dLtR$.
      \item
      Il est $\frF$-non-dégénéré pour un éventail non-singulier $\frF$ raffinant $\frF(\cXo)$.
      \end{itemize}
 \end{Lm}
 \begin{proof}
           C'est une conséquence du lemme~\ref{Lm:polyfaces}.
      \end{proof}

  La proposition suivante est cruciale dans le paragraphe~\ref{Sec:carp}.
 \begin{Pp} \label{Pp:degbase}
              Si $f\in R[\Lambda]$ est non-dégénéré, alors $\Delta(f_k)=\Delta(f)$.
        
            Si $f\in R[\Lambda]$ satisfait $\Delta(f_k)=\Delta(f)$, alors $f$ est non-dégénéré si et seulement si $f_k$ est non-dégénéré.
 \end{Pp}

 \begin{proof}
 La première assertion est claire.
 
 Pour la deuxième, la direction \og seulement si\fg{} vient de la préservation de lissité pendant le changement de base $R\to k$. La direction \og si\fg{} est plus subtile. 
 
Pour tout éventail $\frF$ non-singulier dans $\LtR$. Construisons les schémas $\cP,\cX,D,(\Ts)_{\sigma\in\frF}$ sur $R$ associés au couple $(\frF, f)$ comme ci-dessus.

 Parallèlement, on peut construire les schémas $\cP_k,\cX_k,D_k,((\Ts)_k)_{\sigma\in\frF}$ sur $k$ associés à $(\frF, f_k)$ par la même recette. Ces schémas ne sont pas nécessairement les changements de base des schémas correspondants sur $R$ précédemment construits. Mais pour les $f$ satisfaisant $\Delta(f_k)=\Delta(f)$, ils le sont. 

Supposons que $f$ satisfait $\Delta(f_k)=\Delta(f)$. Prenons un éventail $\frF$ non-singulier et raffinant $\frF(\Delta(f))$. Comme $\frF$ est complet, $\cP$ est propre sur $R$. Ainsi tous les $\cX\cap\overline{\Ts}$ le sont aussi. Ils sont aussi plats sur $R$ par lemme~\ref{Lm:flatallorb}. 

Grâce au lemme suivant, on a des équivalences:\begin{align*}
                    & f_k \text{ est non-dégénéré}\\
         \Equiv & f_k\text{ est }\frF-\text{non-dégénéré}\\ 
          \Equiv &\cX_k\cap\overline{(\Ts)_k}\text{ est lisse sur }k, \text{ pour tout }\sigma\in\frF\\                                                                                                          \Equiv & \text{la fibre spéciale du morphisme }\cX\cap\overline{\Ts}\to\Spec{R}\text{ est lisse pour tout }\sigma\in\frF\\
\Equiv & \cX\cap\overline{\Ts}\to\Spec{R}\text{ est lisse pour tout }\sigma\in\frF.\\
\Equiv & f\text{ est }\frF-\text{non-dégénéré}\\  
\Equiv &  f\text{ est non-dégénéré}.\\
                     \end{align*}
\end{proof}

\begin{Lm}
      Soit $\pi:\cY\to\cZ$ un morphisme propre et plat entre des schémas noethériens. Alors il est lisse si et seulement si toutes ses fibres sur les \emph{points fermés} de $\cZ$ sont lisses.
\end{Lm}

\begin{proof}
    Par~\cite[Corollaire 6.8.7]{EGA4}, le sous-ensemble $S$ de $\cY$ où $\pi$ n'est pas lisse est un fermé de $\cY$. Comme $\pi$ est propre, $\pi(S)$ est un fermé de $\cZ$, et donc, s'il est non-vide, contient un point fermé de $\cZ$, ce qui est exclu par notre hypothèse. Ainsi $\pi(S)$ est vide et $\pi$ est partout lisse.   
    
\end{proof}

Résumons:
\begin{Pp}\label{Pp:divisor}
   Soit $\cXo\hto \rmT$ un diviseur non-dégénéré. Soient $\frF$ un éventail non-singulier raffinant $\frF(\cXo)$, et $\cP$ la variété torique lisse et propre sur $R$ associée à $\frF$. Soit $\cX$ l'adhérence schématique de $\cXo$ dans $\cP$. Notons $D\colonequals\cP-\rmT$. Alors, 
   \begin{itemize}
           \item  L'immersion fermée $\cX\hto\cP$ est un diviseur lisse sur $R$, et $D\hto\cP$ et $D\cap\cX\hto\cX$ sont des diviseurs à croisements normaux stricts relatif à $R$.
           \item Notons $\cL\colonequals\cO(\cX)$ le fibré en droites sur $\cP$. Alors, pour tout $n\geq 0$, le faisceau $\cL^{\otimes n}$ sur $\cP$ est \emph{acyclique} par rapport au foncteur \og section globale\fg{}. Le $R$-module $f^{-n}R[(n\Delta(f)\cap\Lambda)]$ ne dépend pas du choix de $f\in R[\Lambda]$ définissant $\cXo$. Il est canoniquement isomorphe à \[H^0(\cP,\cL^{\otimes n}).\]
           \item Le groupe $\Aut(\Lambda)\ltimes\rmT(R)$ agit naturellement sur $\rmT$. Soit $\frG$ un sous-groupe dont l'action sur $\LtR$ préserve $\frF$, alors $\frG$ est naturellement un sous-groupe de $\Aut(\cP,\cX,D)$. Supposons en plus que $\Lambda\otimes R$ est auto-dual en tant que $R[\frG]$-module. Alors l'hypothèse~\ref{Hyp:Omega_trivial} est vérifiée pour le quadruple $(\cP,\cX,D,\frG)$.             
     \end{itemize}      
\end{Pp}

\begin{proof}  Par le lemme~\ref{Lm:equivnondeg}, le diviseur $\cXo\hto\rmT$ est $\frF$-non-dégénéré, ce qui, par définition, montre la première assertation.

         Soit $f\in R[\Lambda]$ définissant $\cXo$. Par le lemme~\ref{Lm:sigma}, on a, pour tout $\sigma\in\frF$ et tout $n\geq 0$, 
            \begin{align*} 
            \HHH^0(U_{\sigma},\cL^{\otimes n})&=(t^{-\lambda_{f,\sigma}}f)^{-n}R[\dual{\sigma}\cap\Lambda]\\
                                                             &=f^{-n}R[(\dual{\sigma}+n\lambda_{f,\sigma})\cap\Lambda].
             \end{align*}         
          On en déduit que                                   
\[\HHH^0(\cP,\cL^{\otimes n})=f^{-n}R[\bigcap_{\sigma\in\frF}(\dual{\sigma}+n\lambda_{f,\sigma})\cap\Lambda].\]
 Pour $n=0$, \[\bigcap_{\sigma\in\frF}\dual{\sigma}=\{0\};\]
 Pour $n>0$, on a \begin{gather*}\dual{\sigma}+n\lambda_{f,\sigma}=n\cdot B_{f,\sigma},\\ \bigcap_{\sigma\in\frF}B_{f,\sigma}=\Delta(f).\end{gather*}
  Donc pour tout $n\geq 0$, \[\HHH^0(\cP,\cL^{\otimes n})=f^{-n}R[(n\Delta(f))\cap\Lambda].\] Il s'ensuit en particulier que tous les $(\cL^{\otimes n})_{n\geq 0}$ sont engendrés par leurs sections globales. Pour montrer l'annulation de toutes leurs cohomologies supérieures, on peut, grâce à \emph{changement de base propre}, supposer que $R$ est un corps. Alors le \og corollaire\fg{} dans \cite[Section~5]{Fulton} permet de conclure.

Enfin, on a l'isomorphisme canonique \[\cO_{\cP}\otimes\Lambda\ito\Omega^1_{\cP}(\log D)\] qui envoie $\lambda\in\Lambda$ sur $t^{-\lambda}\rmd (t^{\lambda})$. Donc \[\HHH^0(\cP,\Omega^1_{\cP}(\log D))=\Lambda\otimes R,\] lequel est auto-dual en tant que $R[\frG]$-module par la condition imposée sur $\frG$.
\end{proof}

\sssec{Caractéristique 0}
Considérons le cas où $R$ est un corps $F$ algébriquement clos de caractéristique 0.

      \begin{Thm}\label{Thm:carzero} Soit $\cXo\hto \rmT$ un diviseur non-dégénéré défini par $f\in F[\rmT]$. Soit $\frG\subset\Aut(\Lambda)\ltimes\rmT(F)$ un sous-groupe fini fixant $f$, tel que $\rho\colonequals\Lambda\otimes F$ est auto-dual en tant que $F[\frG]$-module. Notons $\Delta\colonequals\Delta(f)$. Alors il existe un (unique) polynôme \[\phi(s)\in\rmK(F[\frG]-\mathrm{mod})[s],\] tel que  \[\sum_{n\geq 0}F[\Lambda\cap (n\Delta)]s^n=\frac{\phi(s)}{(1-s)\det(1-s\rho)}\] dans $\rmK(F[\frG]-\mathrm{mod})[[s]]$, et tel que la cohomologie de de Rham algébrique satisfait \[\Hdr^\bullet(\cXo)=\det(1-\rho)-(-1)^{\dim\rmT}\det(\rho)\phi(1)\] dans $\rmK(F[\frG]-\mathrm{mod})$.                  
        \end{Thm}

       Pour tout polytope $\Delta\subset\LtR$, on appelle $\sum_{n\geq 0}F[\Lambda\cap (n\Delta)]s^n$ la \emph{série d'Ehrhart} associée à $\Delta$.

        \begin{proof}
                 Comme $\frG$ fixe $f$, son action sur $\LtR$ préserve $\frF(\cXo)$. Prenons un éventail non-singulier $\frF$ raffinant $\frF(\cXo)$. Comme $\frG$ est fini, on peut supposer en plus que $\frF$ est préservé par $\frG$.  Prenons $\cP$ la variété torique lisse et propre associée à $\frF$, notons $\cX$ l'adhérence schématique de $\cXo$ dans $\cP$, et notons $D\colonequals\cP-\rmT$. Par la proposition~\ref{Pp:divisor}, on a, pour tout $n\geq 0$,
                 \begin{align*}\HHH^\bullet(\cP,\cL^{\otimes n})&=\HHH^0(\cP,\cL^{\otimes n})\\
                                                        &=f^{-n}F[(n\Delta)\cap\Lambda]\\
                                                        &\ito F[(n\Delta)\cap\Lambda]\end{align*} dans $\rmK(F[\frG]-\mathrm{mod})$, où le dernier isomorphisme utilise le fait que $\frG$ fixe $f$. On en déduit que \[\sum_{n\geq 0}\HHH^\bullet(\cP,\cL^{\otimes n})s^n=\sum_{n\geq 0}F[(n\Delta)\cap\Lambda]s^n\] dans $\rmK(F[\frG]-\mathrm{mod})[[s]]$.
                 
                 Ensuite, comme l'hypothèse~\ref{Hyp:Omega_trivial} est vérifiée, on peut appliquer la proposition~\ref{Pp:Hilbert} pour trouver le polynôme $\phi(s)$ demandé. 
                                  
                 L'assertion concernant la cohomologie de $\cXo$ vient de la proposition~\ref{Pp:coh_poly}. En fait, comme $\cP-D=\rmT$, \[\Hdr^\bullet(\cP-D)=\Hdr^\bullet(\rmT)=\det(1-\rho).\]         
        \end{proof}

\sssec{Caractéristique $p>0$}\label{Sec:carp}
             Supposons le corps de base $k$ algébriquement clos de caractéristique $p>0$. Soit $\rmT$ un tore sur $k$ dont le réseau des poids est $\Lambda$. On utilise la cohomologie étale $\ell$-\emph{adique} où $\ell$ est un nombre premier inversible dans $k$. Désignons par $\CC$ le corps des nombres complexes, et par $\Qb$ son sous-corps des nombres algériques sur $\QQ$.
             
             Dans ce cas, le théorème~\ref{Thm:carzero} reste vrai presque verbatim. Mais il y a quelques modifications à apporter, dûes au fait qu'on utilisera le formalisme de \emph{cycles proches} pour le démontrer.

              \begin{Lm}\label{Lm:fieldext}
                   Soit $\frG$ un groupe fini. Soit $F_1\hto F_2$ une extension de corps. Supposons $F_1,F_2$ algébriquement clos de caractéristique~$0$. Alors le morphisme naturel \[\rmK(F_1[\frG]-\mathrm{mod})\otimes_{F_1}F_2\to\rmK(F_2[\frG]-\mathrm{mod})\] est un isomorphisme.
              \end{Lm}

              \begin{Lm} \label{Lm:autodual}
                   Pour tout sous-groupe fini $\frG\subset\Aut(\Lambda)$ et tout corps $F$ algbébriquement clos de caractéristique 0, $\rho_F\colonequals\Lambda\otimes F$ est auto-dual en tant que $F[\frG]$-module.
                  \end{Lm}

                  \begin{proof}
                           Le corps $F$ contient une copie de $\Qb$, il suffit de démontrer le lemme en supposant $F=\Qb$. Ensuite, au vu du lemme~\ref{Lm:fieldext}, on peut effectuer l'extension de scalaires $\Qb\hto\CC$ et supposer $F=\CC$. 
                           
                           On sait qu'une représentation complexe d'un groupe fini $\frG$ est auto-duale si et seulement si son \emph{caractère} prend valeurs dans le sous-corps $\RR$ des nombres réels. En particulier, toute complexification des représentations \emph{réelles} de $\frG$ est auto-dual. C'est bien le cas dans notre situation, notre représentation étant même définie sur $\ZZ$.
                  \end{proof}

          \begin{Thm}  \label{Thm:carp}
                   Soit $\cXo$ un diviseur non-dégénéré sur $\rmT$, défini par une fonction $f$ sur $T$. Soit $\frG\subset\Aut(\Lambda)$ un sous-groupe fini fixant $f$. Notons $\rho\colonequals\Lambda\otimes\Qlb$, et $\Delta\colonequals\Delta(f)$. Alors il existe un (unique) polynôme \[\phi(s)\in\rmK(\Qlb[\frG]-\mathrm{mod})[s],\] tel que
           \[\sum_{n\geq 0}\Qlb[(n\Delta)\cap\Lambda]s^n=\frac{\phi(s)}{(1-s)\det(1-s\rho)}\] 
           dans $\rmK(\Qlb[\frG]-\mathrm{mod})[[s]]$. De plus, on a, dans $\rmK(\Qlb[\frG]-\mathrm{mod})$,
           \[\Hcet^\bullet(\cXo,\Qlb)=\Het^\bullet(\cXo,\Qlb)=\det(1-\rho)-(-1)^{\dim\rmT}\det(\rho)\phi(1).\]
           \end{Thm}

           \begin{proof}           
           En se restreignant à la clôture algébrique du sous-corps de $k$ engendré par les coefficients de $f$, on peut supposer que $k$ est \emph{dénombrable}.
           
            Ensuite, prenons un anneau local $R$, hensélien et de valuation discrète, de corps résiduel $k$, et dont le corps de fractions est de caractéristique 0. On pourrait prendre par exemple $R=\rmW(k)$ l'\emph{anneau de Witt} associé à $k$. Soit $F$ une clôture algébrique de son corps de fractions. Alors $\card{F}=\card{R}\leq\card{\CC}$. Ainsi on peut prendre un prolongement de corps $\iota_F:F\hto\CC$.
                     
      Prenons $f_R\in R[\Lambda]$ un \emph{relèvement} $\frG$-invariant de $f$ tel que \[\Delta(f_R)=\Delta.\]  Soit $\cXo_R\hto\rmT_R$ le diviseur associé, qui est non-dégénéré par la proposition~\ref{Pp:degbase}. Par changement de base, on obtient un diviseur $\cXo_{\CC}\hto\rmT_{\CC}$, aussi non-dégénéré, défini par $f_{\CC}\colonequals\iota_F(f_R)\in\CC[\Lambda]$. 
      
      Par le lemme~\ref{Lm:autodual}, $\rho_{\CC}\colonequals\Lambda\otimes\CC$ est auto-duale comme représentation de $\frG$. En appliquant le théorème~\ref{Thm:carzero} au triplet \[(\cXo_{\CC}\hto\rmT_{\CC},f_{\CC}\in \CC[\Lambda],\frG\subset\Aut(\Lambda)),\] on trouve un polynôme \[\phi(s)\in\rmK(\CC[\frG]-\mathrm{mod})[s]\] tel que \[\sum_{n\geq 0}\CC[\Lambda\cap (n\Delta)]s^n=\frac{\phi(s)}{(1-s)\det(1-s\rho_{\CC})}\] dans $\rmK(\CC[\frG]-\mathrm{mod})[[s]]$, et tel que   
      \[\Hdr^\bullet(\cXo_{\CC})=\det(1-\rho_{\CC})-(-1)^{\dim\rmT}\det(\rho_{\CC})\phi(1)\] dans $\rmK(\CC[\frG]-\mathrm{mod})$.   
      
     Les représentations $\CC[\Lambda\cap (n\Delta)]$ et $\rho_{\CC}$ de $\frG$ étant naturellement définies sur $\ZZ$, les coefficents de $\phi(s)$ sont donc \og définis sur $\Qb$\fg{}, i.e. appartiennent à $\rmK(\Qb[\frG]-\mathrm{mod})$. 
     
     Comme $\card{\Qlb}=\card{\Ql}=\card{\CC}$, on peut choisir une extension de corps $\iota_{\ell}:\Qlb\hto\CC$. Alors l'inclusion $\Qb\hto\CC$ se factorise canoniquement par un prolongement $\QQ\hto\Qlb$ suivi de $\iota_{\ell}$. 
     
     Alors, le polynôme $\phi(s)$ satisfait \begin{gather*}\sum_{n\geq 0}\Qlb[(n\Delta)\cap\Lambda]s^n=\frac{\phi(s)}{(1-s)\det(1-s\rho)},\\
     \Hdr^\bullet(\cXo_{\Qlb})=\det(1-\rho)-(-1)^{\dim\rmT}\det(\rho)\phi(1),\end{gather*}
     puisqu'il satisfait ces identités après extension de scalaires donnée par $\iota_{\ell}$.                         
       
      Montrons ensuite l'égalité \[\Hdr^\bullet(\cXo_{\Qlb})=\Het^\bullet(\cXo,\Qlb)\] dans $\rmK(\Qlb[\frG]-\mathrm{mod})$. Il suffit de la montrer après l'extension de scalaires $\iota_{\ell}$. On explicitera en fait un quasi-isomorphisme $\Aut_R(\cXo_R)$-équivariant:
       \[\Hdr^\bullet(\cXo_{\CC})\ito\Het^\bullet(\cXo,\Qlb)\otimes_{\Qlb,\iota_{\ell}}\CC.\]
       
       L'existence d'un tel morphisme équivariant vient de la proposition~\ref{Pp:compare} suivante, car $\cXo_R$ est lisse sur $R$. Ce morphisme est un quasi-isomorphisme par la même proposition, car on a une \og bonne\fg{} compactification de $\cXo_R$ décrite dans le paragraphe suivant.
       
         Prenons un éventail non-singulier $\frF$ raffinant $\frF(\cXo)$, Comme $\frG$ est fini, on peut supposer en plus que $\frF$ est préservé par $\frG$. Soit $\cP_R$ la variété torique sur $R$ associée à $\frF$. Prenons $\cX_R$ l'adhérence schématique de $\cXo_R$ dans $\cP$. Par la proposition~\ref{Pp:divisor}, $\cXo_R$ est propre et lisse sur $R$, et le complément de l'ouvert $\cXo_R$ dans $\cX_R$ est un diviseur à croisements normaux stricts relatif à $R$.                              
         
         Traitons finalement $\Hcet^\bullet(\cXo,\Qlb)$. Les $\Qlb[\frG]$-modules $\Qlb[(n\Delta(f))\cap\Lambda]$ et $\rho$ étant définis sur $\ZZ$, sont tous auto-duaux. Ainsi $\phi(1)$, et donc $\Het^\bullet(\cXo,\Qlb)$, sont aussi auto-duaux. Comme $\cXo$ est lisse, $\Hcet^\bullet(\cXo,\Qlb)$, étant dual de $\Het^\bullet(\cXo,\Qlb)$ (à décalage pair près) par la dualité de Verdier, est égal à $\Het^\bullet(\cXo,\Qlb)$ dans $\rmK(\Qlb[\frG]-\mathrm{mod})$.
           \end{proof}

             \begin{Pp}\label{Pp:compare}
                            Soient des anneaux $R,k,F$ comme dans la preuve du théorème~\ref{Thm:carp}. Soit $\cXo_R$ un schéma lisse sur $R$. Alors pour tout prolongement \[\iota_F:F\hto\CC, \iota_{\ell}:\Qlb\hto\CC,\] on a un morphisme canonique \[\Hdr^{\bullet}(\cXo_F)\otimes_{F,\iota_F}\CC\to\Het^\bullet(\cXo_k,\Qlb)\otimes_{\Qlb,\iota_{\ell}}\CC\] compatible à tout automorphisme du $R$-schéma $\cXo_R$.                            
                            
                            De plus, si on dispose d'un schéma $\cX_R$ propre et lisse sur $R$, tel que $\cXo_R$ se prolonge comme un ouvert dans $\cX_R$ de manière à ce que le complément soit un diviseur à croisements normaux stricts relatif à $R$, alosr le morphisme ci-dessus est un quasi-isomorphisme.
                      \end{Pp}

                      \begin{proof}
                         
      Comme $\cXo_R$ est lisse sur $R$, on a par \cite[Exposé XIII, Reformulation 2.1.5]{SGA7-2}, que pour tout groupe abelien $A$ qui est fini en tant qu'ensemble et de torsion premier à $p$,  le cycle proche $\rmR\psi(A_{\cXo_F})$ est isomorphe à $A_{\cXo_k}$. Le morphisme canonique (2.1.8.1) de \emph{loc.cit.} devient alors
      \begin{equation}\label{Eq:specialisa}\Het^\bullet(\cXo_F,A)\to\Het^\bullet(\cXo_k,A).\end{equation}
      Par définition, \begin{gather*}\Het^\bullet(\cXo_F,\Zl)\colonequals\lim_{n\geq 0}\Hcl^\bullet(\cXo_{\mathrm{an}},\ZZ/\ell^n),\\
      \Het^\bullet(\cXo_F,\Qlb)\colonequals\Het^\bullet(\cXo_F,\Zl)\otimes_{\Zl}\Qlb.\end{gather*} Idem pour $\cXo_k$. Ainsi, on déduit de~\eqref{Eq:specialisa} un morphisme canonique \[\Het^\bullet(\cXo_F,\Qlb)\to\Het^\bullet(\cXo_k,\Qlb).\]
  Ce sera un quasi-isomorphisme si $\cXo_R$ ademet une compactification comme décrite dans l'énoncé. En fait, par proposition 2.1.9 de \emph{loc.cit.}, l'existence d'une telle compactification implique que~\eqref{Eq:specialisa} est un quasi-isomorphisme. 
  
  Pour conclure, il reste donc à trouver un isomorphisme canonique \begin{equation}\label{Eq:derhamet}\Hdr^{\bullet}(\cXo_F)\otimes_{F,\iota_F}\CC\ito\Het^\bullet(\cXo_F,\Qlb)\otimes_{\Qlb,\iota_{\ell}}\CC.\end{equation}
  
  Montrons-le.  
  
   Analysons le côté droite de~\eqref{Eq:derhamet}. Notons $\cXo_{\CC}$ le changement de base de $\cXo_F$ par rapport à $\iota_F:F\hto\CC$. Pour tout groupe abélien $A$ comme ci-dessus, l'extension $\iota_F$ induit \[\Het^\bullet(\cXo_F, A)\ito\Het^\bullet(\cXo_{\CC},A),\] lequel, par \cite[Exposé~XI, Théorème~4.4]{SGA4-3} et en notant $\cXo_{\mathrm{an}}$ la variété complexe analytique associée à $\cXo_{\CC}$, est isomorphe à $\Hcl^\bullet(\cXo_{\mathrm{an}},A)$, où $\Hcl^\bullet$ désigne la cohomologie d'un faisceau par rapport à la \emph{topologie classique} d'un \emph{espace topologique}. On en déduit \[\Het^\bullet(\cXo_F,\Qlb)\ito\Hcl^\bullet(\cXo_{\mathrm{an}},\Qlb).\]
 
       Le côté gauche de~\eqref{Eq:derhamet} est isomorphe à $\Hdr^{\bullet}(\cXo_{\CC})$, lequel admet un morphisme naturel, montré quasi-isomorphe dans~\cite{Gro}, vers l'hypercohomologie \[\HH^{\bullet}(\cXo_{\mathrm{an}},\Omega_{\mathrm{an}}^{\bullet})\] où $\Omega_{\mathrm{an}}^{\bullet}$ est le complexe des faisceaux des \emph{formes différentielles holomorphes}. Cette hypercohomologie est quasi-isomorphe à $\Hcl^\bullet(\cXo_{\mathrm{an}},\CC)$ par le lemme de Poincaré, 
       qui dit que le morphisme naturel \[\ul{\CC}\ito\Omega^\bullet_{\mathrm{an}}\] est un quasi-isomorphisme pour la \emph{topologie classique}.           
       \end{proof}

         \begin{Rq}  D'après la preuve, $\frG$ pourrait, plus généralement, être un sous-groupe fini de $\Aut(\Lambda)\ltimes\Lambda\otimes\Unit{R}$ ayant les mêmes propriétés.           \end{Rq}

 \sssec{Trace et volume}        
        Comment calculer explicitement la trace d'un $\gamma\in\frG$ sur $\phi(1)$? 
                
          On a besoin de la notion de \emph{volume} pour les polytopes pour répondre à cette question.

 Pour tout réseau $\Lambda$, $\LtR$, considéré comme un groupe additif muni de sa topologie usuelle, est muni d'une unique \emph{mesure de Haar} telle que la mesure induite sur $(\LtR)/\Lambda$ est de volume totale~1. Pour un polytope $\Delta\subset\LtR$ de dimension maximale, son volume sera mesuré sous cette mesure.
        
        Sinon, en le déplaçant par un élément de $\Lambda$, on peut supposer qu'il contient 0. Alors, prenons $H$ le plus petit sous-espace linéaire de $\LtR$ contenant $\Delta$. Alors $H\cap\Lambda$ est un réseau cocompact dans $H$, ce qui induit une mesure canonique sur $H$ comme ci-dessus. Le volume de $\Delta\subset H$ sera alors mesuré sous cette mesure. En tous cas, nous allons noter le volume de $\Delta$ par $\Vol(\Delta)$.

        \begin{Lm}  \label{Lm:volume}
                 Pour tout polytope (convexe et rationel) $\Delta\subset\LtR$, on a \[\card{(n\Delta)\cap\Lambda}=\Vol(\Delta)n^{\dim\Delta}+O(n^{\dim(\Delta)-1})\] lorsque $n\to+\infty$. Ainsi, \[\sum_{n\geq 0}\card{(n\Delta)\cap\Lambda}s^n\sim\frac{\dim(\Delta)!\Vol(\Delta)}{(1-s)^{\dim(\Delta)+1}}\] lorsque $s\to 1^-$. 
        \end{Lm}

          \begin{Pp} \label{Pp:tracevol}
               Suivons les notations du théorème~\ref{Thm:carzero} (resp. théorème~\ref{Thm:carp}), alors pour tout $\gamma\in\frG$, on a \[\Tr(\gamma;\phi(1))=\det(1-\gamma;(\Delta^{\gamma}-\Delta^{\gamma})^{\perp})\dim(\Delta^{\gamma})!\Vol(\Delta^{\gamma}).\]
               \end{Pp}

               \begin{proof}                     
                Traitons le cas du théorème~\ref{Thm:carzero}. L'autre cas sera similaire.
                
               On utilise l'égalité  \[\sum_{n\geq 0}F[\Lambda\cap (n\Delta)]s^n=\frac{\phi(s)}{(1-s)\det(1-s\rho)}\] dans $\rmK(F[\frG]-\mathrm{mod})[[s]]$. 
               
               Pour $\gamma\in\frG$, remarquons d'abord que \[\Tr(\gamma; F[\Lambda\cap (n\Delta)])=\card{\Lambda\cap (n\Delta)}.\] Appliquant le lemme~\ref{Lm:volume} au polytope $\Delta^{\gamma}$, i.e. au sous-polytope (rationel) des points fixes sous l'action de $\gamma$ sur $\Delta$, on obtient \begin{align*}\Tr(\gamma;\phi(1))\sim&\Tr(\gamma;\phi(s))\\=&(1-s)\det(1-s\gamma;\rho)\sum_{n\geq 0}\card{\Lambda\cap (n\Delta^\gamma)}s^n\\ \sim &\frac{\det(1-s\gamma;\rho)}{(1-s)^{\dim(\Delta^{\gamma})}}\dim(\Delta^{\gamma})!\Vol(\Delta^{\gamma})\\\sim&\det(1-\gamma;(\Delta^{\gamma}-\Delta^{\gamma})^{\perp})\dim(\Delta^{\gamma})!\Vol(\Delta^{\gamma})\end{align*} lorsque $s\to 1^-$. 
               \end{proof}

\sec{Sommes trigonométriques équivariantes} Suivons les notations du paragraphe~\ref{Sec:Euler}.
Soit $k$ un corps algébriquement clos de caractéristique $p>0$, et $\rmT$ un tore (déployé) sur $k$, dont le réseau des poids est $\Lambda$. Soit $f\in k[\Lambda]$ une fonction sur $\rmT$. Notons $\Dinf(f)\subset\LtR$, appelé \emph{polytope de Newton à l'infini} associé à $f$, l'enveloppe convexe de $\{0\}\cup\Delta(f)$. On note $\AA^1\colonequals\Spec(k[c])$ la droite affine sur $k$.

\begin{Def} \label{Def:nondeginf}
     Une fonction $f\in k[\rmT]$ est dite \emph{non-dégénérée à l'infini} si pour toute face $\Gamma$ de $\Dinf(f)$ ne contenant pas~$0$, les conditions suivantes sont satisfaites: \begin{itemize} \item Le diviseur sur $\rmT$ défini par $f|_{\Gamma}$ est lisse. \item Il existe $l\in\dual{\Lambda}$ telle que, considérée comme fonction sur $\Gamma$, elle est constante de valeur (entière) non-divisible par~$p$.\footnote{Cette deuxième condition n'était pas mentionnée dans~\cite{Loeser}. Cette négligeance a été rattrapée dans~\cite{Fu}.}\end{itemize}
\end{Def}

 \begin{Thm}  \label{Thm:trigtore}
           Soit $\cXo$ un diviseur sur $\rmT$ défini par $f\in k[\rmT]$ non-dégénérée à l'infini. Soit $\frG\subset\Aut(\Lambda)$ un sous-groupe fini fixant $f$. Notons $\rho\colonequals\Lambda\otimes\Qlb$. Alors il existe un (unique) polynôme \[\phi(s)\in\rmK(\Qlb[\frG]-\mathrm{mod})[s],\] tel que
           \[\sum_{n\geq 0}\Qlb[(n\Dinf(f))\cap\Lambda]s^n=\frac{\phi(s)}{(1-s)\det(1-s\rho)}\] dans $\rmK(\Qlb[\frG]-\mathrm{mod})[[s]]$. De plus, considérant $f$ comme un morphisme $\rmT\to\AA^1$, on a, dans $\rmK(\Qlb[\frG]-\mathrm{mod})$,
           \[\Hcet^\bullet(\rmT,f^*\Lpsi)=(-1)^{\dim\rmT}\det(\rho)\phi(1).\]
           \end{Thm}

           \begin{proof}
               Soit $\cF$ un faisceau sur $\AA^1$ modérément ramifié à l'infini. Par la \emph{formule de Grothendieck-Ogg-\v Safarevi\v c} telle que décrite dans~\cite[Théorème (2.2.1.2)]{Laumon} et le fait \[\Swan_{\infty}\Lpsi=1\] comme décrit dans l'Exemple (2.1.2.8) de \emph{loc. cit.}, on a \[\Hcet^\bullet(\AA^1,\cF\otimes\Lpsi)=\Hcet^\bullet(\AA^1,\cF)-\cF_{\bar{\eta}}\] dans $\rmK(\Qlb-\mathrm{mod})$, où $\bar{\eta}$ est un point géométrique localisé sur le point générique de $\AA^1$. Supposons qu'un groupe fini $\frG$ agit sur $\cF$, alors l'égalité ci-dessus est vraie même dans $\rmK(\Qlb[\frG]-\mathrm{mod})$. Pour le voir, on peut appliquer l'égalité ci-dessus à chaque composante isotypique de $\cF$, sachant que ces composantes seront aussi modérément ramifiées à l'infini.
                 
       Par~\cite[Theorem~4.2]{Loeser}, on sait que le complexe de faisceaux $\cF\colonequals\rmR f_!\Qlb$ est modérément ramifié à l'infini. Il est muni d'une action du groupe $\frG$ comme dans l'énoncé puisque $f$ est $\frG$-invariante. Il s'ensuit que  \begin{align*}\Hcet^\bullet(\rmT,f^*\Lpsi)=&\Hcet^\bullet(\AA^1,(\rmR f_!\Qlb)\otimes\Lpsi)\\
      =&\Hcet^\bullet(\AA^1,\rmR f_!\Qlb)-(\rmR f_!\Qlb) _{\bar{\eta}}\\
      =&\Hcet^\bullet(\rmT,\Qlb)-\Hcet^\bullet(f^{-1}(\bar{\eta}),\Qlb)\\ 
      =&\det(1-\rho)-\Hcet^\bullet(G^{-1}(0),\Qlb)\end{align*}  dans $\rmK(\Qlb[\frG]-\mathrm{mod})$, où on a noté par $G$ la fonction $f-c$ sur le tore $\rmT\times_k\bar{\eta}$. On a que $G$ est invariante sous le groupe $\frG$, et que $\Delta(G)=\Dinf(f)$. 
      
      La preuve du thèorème~4.2 de \emph{loc. cit.} montre aussi que $G$ est \emph{non-dégénérée}, on peut donc appliquer le théorème~\ref{Thm:carp} sur~$G$. On obtient ainsi le polynôme $\phi(s)$ ayant les deux propriétés décrites dans l'énoncé.                                   
           \end{proof}

           \ssec{Applications}
           Nous étudions ici des sommes trigonométriques sur des \emph{espaces vectoriels}. 
           
            Soit $U$ un espace vectoriel de dimension finie sur $k$. Soit $B$ une $k$-base de $U$. Elle induit une stratification de $U$ par des tore $\Gm^J$, où $J$ parcourt les sous-ensembles de $B$, via \[\Gm^J\hto U, (t_j)_{j\in J}\mapsto \sum_j t_j\cdot j.\] De plus, le réseau des poids et le réseau des copoids du tore $\Gm^B$ sont tous les deux canoniquement isomorphes à $\ZZ^B$.
            
            Le lemme suivant est clair.

                       \begin{Lm}
            Soit $f$ une fonction sur $U$, alors les conditions suivantes sont équivalentes: \begin{itemize} \item La restriction $f|_{\Gm^B}$ est non-dégénérée à l'inifini. \item Pour tout $J\subseteq B$, $f|_{\Gm^J}$ est non-dégénérée à l'inifini. \end{itemize}
            \end{Lm}

                Pour un $k$-espace vectoriel $V$, notons \begin{gather*}\Sym(V)\colonequals\oplus_{n\geq 0} \Sym^n(V),\\ \ST(V)\colonequals\oplus_{n\geq 0} \ST^n(V).\end{gather*} Alors $\ST(\dual{V})$ est le dual de $\Sym(V)$.
                
               Soit $I$ un ensemble fini. Soit $\gamma\in\frS_I$. Notons par $I/\gamma$ l'ensemble des $\gamma$-orbites dans $I$. Pour une telle $\gamma$-orbite $O$, sa longeur sera notée $|O|$.

\sssec{Des fonctions génératrices associées aux permutations}
         Dans l'anneau des séries formelles $\ZZ[[(t_O)_{O\in I/\gamma}]]$, définissons 
\[\cP_{\g}((t_O)_{O\in I/\g})\colonequals (1-\sum_{J\subseteq I/\g, J\neq\emptyset}(-1)^|J|(\sum_{O\in J}|O|-1)\prod_{O\in J}t_O)^{-1}.\] On la développe  comme \[\sum_{m_O\geq 0, \forall O\in I/\g}\cC_{\gamma}((m_O)_O)\prod_O t_O^{m_O},\] où les $\cC_{\gamma}((m_O)_O)\in\ZZ$.

               \begin{Lm} On a
                \[\cC_{\gamma}((m_O)_O)=\sum_{0\leq n_O\leq m_O,\forall O\in I/\gamma}(\sum_O n_O)!\prod_O\frac{(-|O|)^{n_O}\binom{m_O}{n_O}}{n_O!}.\]
               \end{Lm}

               \begin{proof}
                  On a \begin{align*}&\cP_{\g}((t_O)_{O\in I/\g})\\=&\frac{1}{(1+\sum_O\frac{|O|t_O}{1-t_O})}\times\frac{1}{\prod_O (1-t_O)}\\=&\sum_{0\leq n_O\leq m_O,\forall O\in I/\gamma}\frac{(\sum_O n_O)!}{\prod_O n_O!}\frac{(-|O|t_O)^{n_O}}{(1-t_O)^{n_O+1}}.
                   \end{align*}
                   On conclut par l'égalité \[\frac{t_O^{n_O}}{(1-t_O)^{n_O+1}}=\sum_{m_O}\binom{m_O}{n_O}t_O^{m_O}.\]
               \end{proof}                                

\begin{Rq} \label{Rq:coeffs}
Pour certains $\g$, ces coefficients se calculent de façons plus simples. Voici les cas intervenant dans la proposition~\ref{Pp:traceIIthree}.

a). Pour $\g$ transitif, \[\cP_{\g}(t_I)=(1+(|I|-1)t_I)^{-1},\] donc pour tout $m\geq 0$, \[\cC_{\g}(m)=(1-|I|)^m.\]

b). Pour $\g$ induisant deux orbites $\{i\},I-\{i\}$ où $i$ est un élément de $I$, notons $t_{\{i\}}$ par $t_i$, on a \[\cP_{\g}(t_i,t_{I-\{i\}})=(1+(|I|-2)t_{I-\{i\}}-([I|-1)t_i t_{I-\{i\}})^{-1},\] donc pour tout $m\geq 0$, \[\cC_{\g}(m,m)=(|I|-1)^m.\]

c). Pour $\g=\id$, on a \[\cP_{\g}((t_i)_{i\in I})=(1-\sum_{J\subseteq I,|J|\geq 2}(-1)^|J|(|J|-1)\prod_{i\in J}t_i)^{-1}.\] Lorsque |I|=3, on a, d'après le commentaire de Gheorghe Coserea dans~\cite{oeis}, que \[\cC_{\id}(m,m,m)=(-1)^m\rmF_m,\] où $\rmF_m\colonequals\sum_{i=0}^m\binom{m}{i}^3$ sont les \emph{nombres de Franel}. 
\end{Rq}

           \sssec{Cas \textrm{I}}                                                
               Soient $(V_i)_{i\in I}$ des sous-espaces vectoriels de dimensions finies de $V$ tels que $V_i=V_{\gamma(i)}$ pour tout $i$. Ainsi pour toute $\gamma$-orbite $O\subset I$, on peut définir $V_O\subset V$ comme $V_i$ pour un quelconque $i\in O$. Notre $\gamma$ agit sur $\prod_{i\in I} V_i$ via $(v_i)\mapsto (v_{\gamma^{-1}(i)})$. Notons par $\phi$ le morphisme \[\prod_i V_i\to\Sym(V), (v_i)\mapsto\prod(1+v_i).\] Il est invariant sous $\gamma$.

Suivant ces notations, on a:                                
                 \begin{Pp} \label{Pp:traceI}  Pour tout $l\in\ST(\dual{V})$ générique, \[\Tr(\gamma;\Hcet^\bullet(\prod_i V_i,(l\circ\phi)^*\Lpsi))=\cC_{\gamma}((\dim V_O)_O).\]                               
                \end{Pp}

                  \begin{proof}  Il s'agit d'une somme trigonométrique sur $\prod_i V_i$, c'est à dire sur un \emph{espace vectoriel}. On aimerait la calculer en s'appuyant sur le théorème~\ref{Thm:trigtore}. Ce théorème étant sur les tores, on doit d'abord stratifier $\prod_i V_i$ par des tores.
                                    
                  Pour tout $i$, prenons une $k$-base $B_i$ de $V_i$, de sorte que $B_i=B_{\gamma(i)}$ sous l'identification $V_i=V_{\gamma(i)}$. Ceci donne une $k$-base $B_O$ de $V_O$ pour toute $\gamma$-orbite $O\subset I$. Ainsi pour tout $i\in I$, $V_i$ est stratifié par les tores $T_{J_i}\colonequals\Gm^{J_i}$ où $J_i$ parcourt les sous-ensembles de $B_i$. Le produit $\prod_i V_i$ est stratifié par les $\prod_i T_{J_i}$. D'ailleurs, le réseau des poids et le réseau des copoids de $\prod_i T_{B_i}$ sont tous les deux canoniquement isomorphes à $\prod_i\ZZ^{B_i}$.                  
                   
                   Pour $l\in\ST(\dual{V})$ générique, considérons la fonction $l\circ\phi$ restreinte au tore $\prod_i T_{B_i}\subset \prod_i V_i$. Le polytope de Newton de cette restriction est $\prod_i\Delta(\{0\}\cup B_i)$ où $\Delta(\{0\}\cup B_i)$ désigne l'enveloppe convexe de $\{0\}\cup B_i$ dans $\RR^{B_i}=\sum_{b\in B_i}\RR b$. Ce polytope contient 0, donc coincïde avec le polytope de Newton à l'infini de cette restriction. Remarquons en passant que pour ce genre de polytope, la preuve de~\cite[Theorem~4.2]{Loeser} montre que la première des deux conditions suivantes implique la deuxième:\begin{itemize}\item Pour tout $l\in\ST(\dual{V})$ générique, la fonction $(l\circ\phi)|_{\prod_i T_{B_i}}$ est non-dégénérée à l'infini.\item Pour tout $l\in\ST(\dual{V})$ générique, la fonction $(l\circ\phi)|_{\prod_i T_{B_i}}$ est non-dégénérée.\end{itemize} Ce polytope est un produit de simplexes. Ainsi ses faces sont produits de sous-simplexes, i.e. sont de la forme \[\prod_{i\in I_1}\Delta(\{0\}\cup J_i)\times\prod_{i\in I_2}\Delta(J_i)\] où  $\alpha:I=I_1\sqcup I_2$ est une partition et les $J_i\subseteq B_i$ pour tout $i\in I$. 
                   
                   Montrons que $(l\circ\phi)|_{\prod_i T_{B_i}}$ est non-dégénérée pour $l$ générique. Concrètement, d'après la description des faces ci-dessus, il faut, la deuxième condition de la définition~\ref{Def:nondeginf} étant évidemment vérifiée, montrer que pour toute partition $\alpha:I=I_1\sqcup I_2$ et pour tout $l\in\ST(\dual{V})$ générique, tous les morphismes \[\phi_{\alpha}:\prod_{i\in I}T_{J_i}\to\Sym(V), (t_i)\mapsto \prod_{i\in I_1}(1+t_i)\times\prod_{i\in I_2}t_i\] sont tels que: $l\circ\phi_{\alpha}$ définit un diviseur lisse sur $\prod_{i\in I}T_{J_i}$.
                                      
                  On peut compléter les morphismes ci-dessus en des diagrammes commutatifs \[\begin{array}{ccc} 
          \prod_{i\in I}T_{J_i} &\overset{\phi_{\alpha}}{\to} & \Sym(V)-\{0\}\\
                      \pi_{\alpha}\da &   & \da    \\
\prod_{i\in I_1}\PP(k\oplus V_{J_i})\times\prod_{i\in I_2}\PP(V_{J_i})&\overset{m}{\to}&\PP(\Sym(V)),
           \end{array}\]  où $\pi_{\alpha}$ désigne le produit des morphismes lisses \begin{gather*} T_{J_i}\to\PP(k\oplus V_{J_i}), t\mapsto k\cdot (1+t),\text{ pour tout } i\in I_1;\\T_{J_i}\to\PP(V_{J_i}), t\mapsto k\cdot t,\text{ pour tout } i\in I_2,\end{gather*} et $m$ est donné simplement par la multiplication. 
           
           On sait que $m$ satisfait~\eqref{Eq:dimension} par proposition~\ref{Pp:ThomI}. Ainsi, le lemme~\ref{Lm:duallisse} montre la non-dégénérescence à l'infini souhaitée. 
           
           Notre $\gamma$ induit une permutation des tores $T=\prod_i T_{J_i}$. Le polytope de Newton, ainsi que celui à l'infini, de $(l\circ\phi)|_T$ sont tous les deux donnés par \[\Delta=\prod_{i\in I}\Delta(\{0\}\cup J_i)\subset\Lambda_T\otimes\RR.\]
                      
           Un tore $T$ comme ci-dessus est préservé par $\gamma$ si et seulement s'il existe des $J_O\subseteq B_O$ pour tout $O\in I/\gamma$ tels que $J_i=J_O$ pour tout $i$ dans l'orbite~$O$.
           
           Pour un tel $T$ $\gamma$-stable, on a \begin{gather*}\dim T=\sum_{i\in I}|J_i|=\sum_{O\in I/\gamma}|O|\cdot|J_O|,\\ \det(\gamma,\Lambda_T)=(-1)^{\sum_O|J_O|(|O|-1)}.\end{gather*} Le sous-polytope $\Delta^\gamma$ des points fixes par $\gamma$ est l'image de l'application \og diagonale\fg{} \[\prod_O\Delta(\{0\}\cup J_O)\hto\prod_{i\in I}\Delta(\{0\}\cup J_i).\] On a donc \begin{gather*} \dim(\Delta^\gamma)=\sum_O|J_O|,\\\Vol(\Delta^\gamma)=\prod_O\Vol(\Delta(\{0\}\cup J_O))=\prod_O\frac{1}{|J_O|!},\\         
           \det(1-\gamma;(\Delta^\gamma-\Delta^\gamma)^{\perp})=\prod_O|O|^{|J_O|},\end{gather*} où \[(\Delta^\gamma-\Delta^\gamma)^{\perp}\colonequals\{u\in\dual{(\Lambda_T)}\otimes\RR|u\text{ est constante sur }\Delta^\gamma\}.\]           
            La trace \[\Tr(\gamma;\Hcet^\bullet(\prod_i V_i,(l\circ\phi)^*\Lpsi))\] est la somme des \[\Tr(\gamma;\Hcet^\bullet(T,(l\circ\phi)^*\Lpsi))\] où $T$ parcourt les strates $\gamma$-stables comme ci-dessus de $\prod_i V_i$. Si $l\in\ST(\dual{V})$ est générique, comme on a déjà montré que $(l\circ\phi)|_T$ est non-dégénérée à l'infini, le théorème~\ref{Thm:trigtore} et la proposition~\ref{Pp:tracevol} impliquent que \begin{align*}&\Tr(\gamma;\Hcet^\bullet(T,(l\circ\phi)^*\Lpsi))\\=&(-1)^{\dim T}\det(\gamma;\Lambda_T)\det(1-\gamma;(\Delta^\gamma-\Delta^\gamma)^{\perp})\dim(\Delta^\gamma)!\Vol(\Delta^\gamma)\\=&(\sum_O|J_O|)!\prod_O\frac{(-|O|)^{|J_O|}}{|J_O|!}.\end{align*} En somme, \begin{align*}&\Tr(\gamma;\Hcet^\bullet(\prod_i V_i,(l\circ\phi)^*\Lpsi))\\=&\sum_{J_O\subseteq B_O,\forall O\in I/\gamma}(\sum_O|J_O|)!\prod_O\frac{(-|O|)^{|J_O|}}{|J_O|!}\\=&\sum_{0\leq n_O\leq \dim(V_O),\forall O\in I/\gamma}(\sum_O n_O)!\prod_O\frac{(-|O|)^{n_O}\binom{\dim V_O}{n_O}}{n_O!}\\=&\cC_{\gamma}((\dim V_O)_O).\end{align*}                                                                                                                                                                                                                          
         \end{proof}

                  \sssec{Cas \textrm{II}}
                           Soit $X$ une courbe projective lisse irréductible sur~$k$. Soit $\cL$ un fibré en droites sur $\cX$. Soient $(\cL_i)_{i\in I}$ des sous-faisceaux de $\cL$ de sorte que $\cL_i=\cL_{\gamma(i)}$ pour tout $i\in I$. Pour toute $\gamma$-orbite $O\subseteq I$, notons $\cL_O\subseteq\cL$ comme $\cL_i$ pour un quelconque $i\in O$. Prenons le morphisme \[\phi:\prod_i\HX{\cL_i}\to\oplus_{n\geq 0} \HX{\cL^{\otimes n}}, (s_i)\mapsto\prod(1+s_i).\]

                           \begin{Pp}\label{Pp:TraceII} Pour tout $l\in\oplus_{n\geq 0}\HHH^1(X,\cL^{\otimes (-n)}\otimes\Omega)$ générique, \[\Tr(\gamma;\Hcet^\bullet(\prod_i\HX{\cL_i},(l\circ\phi)^*\Lpsi)=\cC_{\gamma}((\dim\HX{\cL_O})_O).\] 
                           \end{Pp}

                \begin{proof} On peut supposer que tous les $\cL_i$ sont non-nuls. Ils sont alors aussi des fibrés en droites sur $X$.
                
                 Le morphisme $\phi$ se factorise comme \[\prod_i\HX{\cL_i}\to\Sym(\HX{\cL}),(v_i)\mapsto\prod(1+v_i),\] suivi de \[\Sym(\HX{\cL})\to\oplus_{n\geq 0}\HX{\cL^{\otimes n}}\] lequel est le produit des morphismes \[\Sym^n(\HX{\cL})\to\HX{\cL^{\otimes n}},\text{ pour tout } n\geq 0.\] Ainsi notre énoncé est du même type que celui de la proposition~\ref{Pp:traceI}, sauf que notre fonctionnelle linéaire $l$ est plus spécifique.
                 
                     Pour adapter la preuve de~\ref{Pp:traceI}, il faut, pour tout $i\in I$, bien choisir une base de $\HX{\cL_i}$. Pour tout $i$, prenons un ensemble fini $B_i$ des points deux-à-deux disctincts de $X$ tels que, \begin{gather*}\dim\HX{\cL_i}=|B_i|,\\ \HX{\cL_i(-\sum_{b\in B_i}b)}=0.\end{gather*} De plus, on peut bien sûr supposer $B_i=B_{\gamma(i)}$ pour tout $i$.
                                          
                     Pour tout $i$ et tout $J_i\subseteq B_i$, on note $\cL_{i, J_i}\colonequals\cL_i(-\sum_{b\in B_i-J_i}b)$ le fibré en droites sur $X$ muni de l'injection faisceautique $\cL_{i,J_i}\hto\cL_i$. Alors, pour tout $b\in B_i$, $\HX{\cL_{i,\{b\}}}$ est de dimension~1. On choisit un élément non-nul $e_b$ là-dedans. Alors, pour tout $J_i\subseteq B_i$,  $\HX{\cL_{i,J_i}}$ est le sous-espace vectoriel de $\HX{\cL_i}$ librement engendré par $(e_b)_{b\in J_i}$. En particulier, $(e_b)_{b\in B_i}$ forment une base de $\HX{\cL_i}$. Il s'ensuit une stratification de $\HX{\cL_i}$ par les tores $T_{J_i}\colonequals\Gm^{J_i}$ où $J_i$ parcourt les sous-ensembles de $B_i$. Le produit $\prod_i\HX{\cL_i}$ est alors stratifé par les tores $\prod_i T_{J_i}$.
                     
               Comme dans la preuve de la proposition~\ref{Pp:traceI}, on est réduit à regarder, pour toute partition $\alpha: I=I_1\sqcup I_2$, des diagrammes commutatifs \[\begin{array}{ccc} 
          \prod\limits_{i\in I}T_{J_i} &\overset{\phi_{\alpha}}{\to} & \mathop{\oplus}\limits_{n\geq 0}\HX{\cL^{\otimes n}}-\{0\}\\
                      \pi_{\alpha}\da &   & \da    \\
\prod\limits_{i\in I_1}\PP(k\oplus \HX{\cL_{i,J_i}})\times\prod\limits_{i\in I_2}\PP(\HX{\cL_{i,J_i}})&\overset{m}{\to}&\PP(\mathop{\oplus}\limits_{n\geq 0}\HX{\cL^{\otimes n}}),
           \end{array}\]  où $\phi_{\alpha}$ envoie $(t_i)$ sur $\prod_{i\in I_1}(1+t_i)\times\prod_{i\in I_2}t_i$, $\pi_{\alpha}$ désigne le produit des morphismes lisses \begin{gather*} T_{J_i}\to\PP(k\oplus \HX{\cL_{i,J_i}}), t\mapsto k\cdot (1+t),\text{ pour tout } i\in I_1,\\T_{J_i}\to\PP(\HX{\cL_{i,J_i}}), t\mapsto k\cdot t,\text{ pour tout }i\in I_2,\end{gather*} et $m$ est donné simplement par la multiplication. 
           
           On sait que $m$ satisfait~\eqref{Eq:dimension}, cette fois par théorème~\ref{Thm:ThomII}. Ainsi, le lemme~\ref{Lm:duallisse} montre que pour tout $l\in\oplus_{n\geq 0}\HHH^1(X,\cL^{\otimes (-n)}\otimes\Omega)$ générique, $l\circ\phi_{\alpha}$ définit un diviseur lisse sur $\prod_{i\in I}T_{J_i}$, ce qui nous permet de conclure comme dans la preuve de~\ref{Pp:traceI}.                     
                \end{proof}


\begin{thebibliography}{6}
\bibitem{Arthur} J. Arthur, Unipotent automorphic representations : conjectures, Astérisque, tome 171-172 (1989), p. 13-71
\bibitem{Atiyah} M. F. Atiyah, R. Bott and L. G\aa rding, Lacunas for Hyperbolic Differential Operators with Constant Coefficients. II, Acta Mathematica 131(1):145-206, December 1973.
\bibitem{SGA4-3} M. Artin, A. Grothendieck, J. L. Verdier, P. Deligne, B. Saint-Donat, Théorie des topos et cohomologie étale des schémas. Tome 3, Séminaire de Géométrie Algébrique du Bois-Marie 1963–1964 (SGA 4). Lecture Notes in Mathematics 305, Springer 1973. 
\bibitem{Boardman} J. M. Boardman, Singularities of differentiable maps, Publications mathématiques de l’I.H.É.S., tome  33 (1967), p. 21-57.
\bibitem{Braverman} A. Braverman, D. Gaitsgory, Geometric Eisenstein series, Inv. Math. 150 (2002), 287-384.
\bibitem{Collingwood} D. Collingwood, W. McGovern, Nilpotent orbits in semisimple Lie algebras, Van Nostrand Reinhold Mathematics Series, Van Nostrand Reinhold Co., New York, 1993.
\bibitem{SGA7-2} P. Deligne, N. Katz, SGA7 t. II. Groupes de monodromie en géométrie algébrique, Lecture Notes in Mathematics, Volume: 340, Springer-Verlag, 1973.
\bibitem{Franel} J. Franel, Sur une question de Laisant. L'intermédiaire des mathématiciens 1, 45-47, 1894.
\bibitem{Fu} L. Fu, $\ell$-adic GKZ hypergeometric sheaves and exponential sums, Advances in Mathematics, Volume 298, 6 August 2016, Pages 51-88.
\bibitem{Fulton} W. Fulton, Introduction to Toric Varieties, (AM-131), Princeton University Press, 1993.                     
\bibitem{Gan} W. T. Gan, N. Gurevich, D. Jiang, Cubic unipotent Arthur parameters and multiplicities of square-integrable automorphic forms, Inventiones mathematicae 149(2),  225-265, August 2002.
\bibitem{Gro}  A. Grothendieck, On  the  de  Rham  cohomology  of  algebraic  varieties, Pub. math. IHES 29 (1966), 95-103.
\bibitem{EGA4}  A. Grothendieck, Éléments de géométrie algébrique : IV. Étude locale des schémas et des morphismes de schémas, Seconde partie, Publications Mathématiques de l'IHÉS, Tome 24 (1965), p. 5-231.
\bibitem{Langlands}  R. P. Langlands, On the Functional Equations Satisfied by Eisenstein Series,  Springer–Verlag Lecture Notes in Maths., vol. 544, Springer–Verlag, Berlin–Heidelberg, New York, 1976, pp. 1–337.
\bibitem{Laumon}  G. Laumon, Transformation de Fourier, constantes d'équations fonctionnelles et conjecture de Weil, Publications Mathématiques de l'IHÉS, Tome 65 (1987), p. 131-210.
\bibitem{Loeser} J. Denef, and F. Loeser, Weights of exponential sums, intersection cohomology, and Newton polyhedra, Inventiones mathematicae 106.2 (1991): 275-294.
\bibitem{Lysenko} V. Lafforgue, S. Lysenko, Geometrizing the minimal representations of even orthogonal groups, Reprensentation Theory,
	An Electronic Journal of AMS,
	Volume 17, Pages 263–325 (May 28, 2013)
	S 1088-4165(2013)00431-4, 
	errata: page webe de Sergey Lysenko.
\bibitem{oeis}\url{https://oeis.org/A000172}	
\bibitem{Stapledon} A. Stapledon, Representations on the cohomology of hypersurfaces and mirror symmetry, Advances in Mathematics 226 (2011) 5268–5297.
\bibitem{Thom} R. Thom, Les singularités des applications différentiables, Annales de l’institut Fourier, tome  6 (1956), p. 43-87.
\bibitem{Zagier} D. Zagier, Integral solutions of Apéry-like recurrence equations, Groups and Symmetries: From the Neolithic Scots to John McKay, CRM Proceedings and Lecture Notes, Vol. 47 (2009), Amer. Math. Society, 349-366.

\end{thebibliography}
\end{document}